\documentclass[letterpaper,11pt]{article}
\usepackage[T1]{fontenc}
\usepackage[utf8]{inputenc}
\usepackage[english]{babel}
\usepackage{fullpage}
\usepackage{amsmath,amsfonts,amssymb,amsthm,mathtools,dsfont,bbm}
\usepackage{graphicx,adjustbox,color,xcolor}
\usepackage{tikz,pgfplots,subcaption}
\usepackage{algorithm,algpseudocode}
\usepackage{booktabs,threeparttable,array,pifont} 
\usepackage{enumitem}
\usepackage{cite}
\usepackage{hyperref}
\usepackage[capitalize,nameinlink,noabbrev]{cleveref}

\allowdisplaybreaks
\mathtoolsset{centercolon}
\usetikzlibrary{arrows.meta}
\pgfplotsset{compat=1.18}
\hypersetup{colorlinks, hypertexnames=false, pageanchor=true,
            linkcolor=blue, citecolor={green!50!black}, urlcolor=cyan,
            pdftitle={On graphs with finite-time consensus and their use in gradient tracking}
            }

\newtheorem{definition}{{Definition}}[section]
\newtheorem{assumption}{{Assumption}}[section]
\newtheorem{theorem}{{Theorem}}[section]
\newtheorem{lemma}[theorem]{{Lemma}}
\newtheorem{corollary}[theorem]{{Corollary}}
\newtheorem{proposition}[theorem]{{Proposition}}
\crefname{assumption}{Assumption}{Assumptions}

\newcommand{\reals}                  {\mathbb R}
\newcommand{\natint}                 {\mathbb N}
\newcommand{\symm}                   {\mathbb S}
\newcommand{\ones}                   {\mathds{1}}
\newcommand{\jm}                     {\hat{\mathrm{\jmath}}}

\newcommand{\tran}                   {^{\mathsf T}}               
\newcommand{\conj}                   {^{\mathsf H}}               
\DeclareMathOperator{\diag}          {diag}                       
\DeclareMathOperator{\col}           {col}                        

\DeclareMathOperator{\dom}           {{\bf dom}}                  
\DeclareMathOperator{\intr}          {{\bf int}}                  

\newcommand{\Prob}                   {\mathbb P}                  
\newcommand{\Ex}                     {\mathbb E}                  

\newcommand{\mini}                   {\text{\upshape minimize}}
\newcommand{\eg}                     {{\it e.g.}}
\newcommand{\ie}                     {{\it i.e.}}

\newcommand{\oline}[1]{\mkern 1.5mu\overline{\mkern-1.5mu#1}}
\newcommand{\F}{\boldsymbol{F}}
\newcommand{\I}{\boldsymbol{I}}
\newcommand{\W}{\boldsymbol{W}}
\newcommand{\f}{\boldsymbol{f}}
\newcommand{\g}{\boldsymbol{g}}
\newcommand{\s}{\boldsymbol{s}}
\newcommand{\x}{\boldsymbol{x}}
\newcommand{\cE}{\mathcal{E}}
\newcommand{\cG}{\mathcal{G}}
\newcommand{\cF}{\mathcal{F}}
\newcommand{\cN}{\mathcal{N}}
\newcommand{\cV}{\mathcal{V}}
\newcommand{\sbar}{\oline{s}}
\newcommand{\xbar}{\oline{x}}
\newcommand{\ftilde}{\tilde{f}}
\newcommand{\ehat}{\hat{e}}

\newcommand{\define}{\triangleq} 
\newcommand{\bxi}{\boldsymbol{\xi}}

\newcolumntype{L}{>{\centering\arraybackslash}m{3cm}}

\title{On graphs with finite-time consensus \\
and their use in gradient tracking\footnote{Author's final version. Accepted for publication in \textit{SIAM Journal on Optimization}.}}
\author{Edward Duc Hien Nguyen%
\thanks{Department of Electrical and Computer Engineering, Rice University. Email: \textsf{en18@rice.edu}.}
\and Xin Jiang%
\thanks{School of Operations Research and Information Engineering, Cornell University. Email: \textsf{xjiang@cornell.edu}. Much of the work was done when XJ was at Lehigh University.}
\and Bicheng Ying%
\thanks{Google Inc. Email: \textsf{ybc@google.com}.}
\and C\'esar A. Uribe%
\thanks{Department of Electrical and Computer Engineering and Ken Kennedy Institute, Rice University. Email: \textsf{cauribe@rice.edu}.}}
\date{January 29, 2025}

\begin{document}

\maketitle

\begin{abstract}
This paper studies sequences of graphs satisfying the finite-time consensus property (\ie, iterating through such a finite sequence is equivalent to performing global or exact averaging) and their use in the decentralized optimization algorithm Gradient Tracking. For each of the studied graph sequences, we provide an explicit weight matrix representation and prove their finite-time consensus property. Moreover, we incorporate such topology sequences into Gradient Tracking and present a new algorithmic scheme called Gradient Tracking for Finite-Time Consensus Topologies (GT-FT). We analyze the new scheme for nonconvex problems with stochastic gradient estimates. Our analysis shows that the convergence rate of GT-FT does not depend on the heterogeneity of the agents' functions or the connectivity of any individual graph in the topology sequence. Furthermore, owing to the sparsity of the graphs, GT-FT requires lower communication costs than Gradient Tracking using the static counterpart of the topology sequence.
\end{abstract}

\section{Introduction} \label{sec:introduction}

We study the optimization problem
\begin{equation} \label{eq:prob}
    \underset{x \in \reals^d}{\mini} \ \ f(x) \define \frac{1}{n} \sum_{i=1}^n f_i (x), \quad \text{where} \ f_i(x) \define \Ex_{\xi_i}[F_i(x; \xi_i)],
\end{equation}
each $f_i \colon \reals^d \rightarrow \reals$ is a smooth, possibly nonconvex function, and the symbol~$\Ex_{\xi_i}$ denotes the expected value of the random variable $\xi_i$ associated with the probability space $\{\Omega_i, \cF_i, \Prob_i\}$. Hence, $f_i$ is defined as the expected value of some loss function $F_i(\cdot, \xi_i)$ over $\xi_i$. Algorithms for solving this type of problems often use stochastic gradients of~$f_i$, and are performed in a decentralized manner; \ie, each function~$f_i$ and probability space $\{\Omega_i, \cF_i, \Prob_i\}$ are held exclusively and privately by an agent $i \in \{1, \ldots, n\}$. Consequently, in decentralized algorithms for solving~\eqref{eq:prob}, agents must communicate with one another according to some network topology. We are motivated by the application of decentralized optimization algorithms to the high-performance computing scenario in which computing resources can be abstracted as agents. Under this scenario, not only is the network topology robust but it is also flexible. Therefore, it is critical to select topologies that reduce communication costs~\cite{lan2017communicationefficient, assran2019stochastic, NEURIPS2021_74e1ed8b, ding2023dsgd}.
\par
In this paper, we address the trade-off between communication costs and convergence rate in decentralized optimization algorithms and propose to use sequences of deterministic topologies that satisfy the \textit{finite-time consensus} property. Topology sequences with the finite-time consensus property have the desirable feature that iterating through the entire graph sequence is equivalent to performing global or exact averaging. Moreover, each of the individual graphs in such a sequence is typically sparse \cite{NEURIPS2021_74e1ed8b, Shi_2016, takezawa2023exponential} and, when used in a decentralized algorithm, requires limited communication costs at each iteration. We study several classes of topology sequences for which this seemingly restrictive requirement holds, including the static de Bruijn graph \cite{Bruijn1946ACP,delvenne2009optimal}, one-peer hyper-cubes \cite{Shi_2016}, one-peer exponential graphs \cite{NEURIPS2021_74e1ed8b}, and $p$-peer hyper-cuboids.
\par
Despite the existence of various topology sequences with the finite-time consensus property, directly incorporating them into decentralized algorithms is not straightforward. Classical analyses of decentralized optimization algorithms (\eg, \cite{alghunaim2021unified, koloskova2021improved,alghunaim2023enhanced}) assume, for example, symmetry and strong connectivity of the mixing matrices, which certain elements in a topology sequence with the finite-time consensus property might violate. Hence, novel analysis for decentralized algorithms using finite-time consensus must be developed.

\paragraph{Contributions}
The contribution of this work is two-fold. First, we study several sequences of graphs that satisfy the finite-time consensus property. For one-peer exponential graphs (with $n=2^\tau$) in \cite{NEURIPS2021_74e1ed8b}, we leverage their circulant property to develop an alternative proof for the finite-time consensus property. For an arbitrary number of agents, we present the sequence of graphs called \textit{$p$-peer hyper-cuboids} and establish their finite-time consensus property. We also show that in certain cases, $p$-peer hyper-cuboids are permutation equivalent to the well-studied de Bruijn graphs.
\par
Moreover, we incorporate the studied topology sequences into the Gradient Tracking (GT) algorithm and present a new algorithmic scheme called Gradient Tracking for Finite-Time Consensus Topologies (GT-FT). We present convergence analysis for GT-FT, with further stepsize tuning and a simple warm-up technique. The analysis establishes that the convergence rate is independent of the connectivity of any of the individual graphs used in the algorithm. Furthermore, we show that GT-FT using the presented graph sequences with finite-time consensus has the same iteration complexity as GT using the static counterparts. Considering the decentralized manner of the algorithms, it suggests that GT-FT has significantly lower communication costs compared with the static counterpart.

\paragraph{Outline}
The rest of the paper is organized as follows. \Cref{sec:related} discusses prior work on decentralized optimization algorithms and existing sequences of topologies satisfying the finite-time consensus property. \Cref{sec:fin-time-cons} formally defines the finite-time consensus property and presents various topology sequences with this property. \Cref{sec:alg-description} includes a description of the Time-Varying Gradient Tracking (TV-GT) algorithm and our modified version: Gradient Tracking for Finite-Time Consensus Graphs (GT-FT). \Cref{sec:analysis} presents convergence analysis of GT-FT under the nonconvex setting where agents only have access to stochastic gradient estimates. Numerical experiments in \Cref{sec:experiments} verify the finite-time consensus property of the topology sequences studied in \Cref{sec:fin-time-cons} and the algorithm analysis in \Cref{sec:analysis}.

\section{Related work} \label{sec:related}

In this section, we first review various decentralized optimization algorithms and describe the scope and settings for their analysis. We highlight various works analyzing gradient tracking methods and detail why their analysis cannot be straightforwardly extended to incorporate sequences of topologies that satisfy the finite-time consensus property. We then discuss prior work that studied sequences of topologies satisfying the finite-time consensus property.
\par
The Decentralized Stochastic Gradient method (DSGD)~\cite{lopes2007diffusion, cattivelli2010diffusion, ram2010distributed} is arguably the most popular and widely studied decentralized algorithm due to its simplicity. However, current analysis of DSGD relies on a bounded heterogeneity assumption, which bounds the allowed difference between local functions/data. The effect of heterogeneity is further magnified by large and sparse topologies, which further degrades the performance of DSGD~\cite{koloskova2020unified}.
\par
Exact or bias-corrected decentralized algorithms have been proposed to circumvent the negative effects caused by heterogeneity on the convergence of DSGD. Examples of these algorithms include EXTRA~\cite{shi2015extra}, Exact Diffusion~\cite{yuan2019exactdiffI}, and gradient tracking methods~\cite{nedic2017achieving}. These algorithms have been studied in a unified framework in~\cite{alghunaim2019decentralized} and~\cite{alghunaim2021unified} without the bounded heterogeneity assumption.
\par
The convergence of EXTRA and Exact Diffusion has only been established for symmetric and static mixing matrices~\cite{alghunaim2021unified}. On the contrary, gradient tracking methods with time-varying topologies (TV-GT) have been studied under various assumptions. Nedi{\'c}~et~al.~\cite{nedic2017achieving} analyze TV-GT, which they call DIGing, for the strongly convex scenario, with true gradient evaluation and with $\tau$-connected graphs. The assumption of $\tau$-connected graphs means that the union of a sequence of $\tau$-length graphs is connected, and is weaker than assuming that the topology at each iteration is connected. Another example of TV-GT is the algorithm NEXT~\cite{di2016next} (and its extension SONATA~\cite{scutari2019distributed}), which are analyzed for the nonconvex composite optimization problems with true gradient evaluation and with $\tau$-connected graphs. We note that DIGing~\cite{nedic2017achieving}, NEXT~\cite{di2016next}, and SONATA~\cite{scutari2019distributed} do not cover the case in which agents can only have access to stochastic gradient estimates of their local objective functions. However, extending the aforementioned analyses to account for stochastic gradient estimates is nontrivial and leads to a new line of research. The more recent analyses performed in~\cite{alghunaim2021unified, koloskova2021improved} successfully analyze Gradient Tracking using stochastic gradient estimates, but their analyses are limited to only considering static graphs. In particular, it is explicitly stated in~\cite{koloskova2021improved} that the extension of their analysis to the time-varying setting is nontrivial. The analysis in \cite{alghunaim2021unified, koloskova2021improved} relies on the eigendecomposition of the weighted adjacency matrix associated with the considered static graph. Thus, it is assumed in \cite{alghunaim2021unified, koloskova2021improved} that the underlying graph is undirected, which does hold for, \eg, one-peer exponential graphs considered in our setting. Even if all the graphs in our proposed sequence were undirected, the product of two symmetric weighted adjacency matrices is not necessarily symmetric. Therefore, the analyses in \cite{alghunaim2021unified, koloskova2021improved} remain inapplicable. Later, Song~et~al.~\cite{song2023communicationefficient} provides convergence analysis for TV-GT under the nonconvex and stochastic setting and makes no assumption on the symmetry of the mixing matrices. Nevertheless, their convergence results rely on the smallest connectivity (in expectation) of the set of all (random) topologies used in the algorithm. This does not cover deterministic sequences of topologies that satisfy the finite-time consensus property where certain elements of the sequence are always disconnected. In summary, existing analyses of gradient tracking are limited in their capacity to encapsulate both stochastic gradients and time-varying graphs, and thus cannot be readily extended to capture our setting (graph sequence with finite-time consensus).
\par
Graph sequences with the so-called \textit{finite-time consensus} have been proposed and analyzed in various contexts. Delvenne~et~al.~\cite{delvenne2009optimal} establish the finite-time property for a sequence of static de Bruijn graphs. Shi~et~al.~\cite{Shi_2016} justify the finite-time consensus property for a sequence of one-peer hyper-cubes (when $n=2^\tau$ for some $\tau \in \natint_{\geq 1}$). Assran~et~al.~\cite{assran2019stochastic} observe in numerical experiments that one-peer exponential graphs (with $n=2^\tau$) have the finite-time consensus property, which is later theoretically justified in \cite{NEURIPS2021_74e1ed8b}. More recently, Takezawa~et~al.~\cite{takezawa2023exponential} claim to build graph sequences with finite-time consensus of any node size $n \in \natint$,  but they do not provide any theoretical justification. Although numerical evidence provided in \cite{takezawa2023exponential} demonstrates the finite-time consensus property of their proposed graphs, no explicit matrix representation is given and further investigation is needed. In another line of research, Ding~et~al.~\cite{ding2023dsgd} extend the optimal message passing algorithm developed in \cite{bar1993optimal} and propose a communication-optimal exact consensus algorithm. The algorithm proposed in \cite{ding2023dsgd} requires an additional copy of the optimization variable at each agent, and with the help of these auxiliary variables, achieves ``finite-time consensus'' for an arbitrary number of agents. The discussion of this approach is out of the scope of this paper, and further investigation is left as future work. \Cref{table:fin-con-topo} summarizes the existing graph sequences for which finite-time consensus is proven to hold. Despite the usefulness of such graph sequences, incorporating them in decentralized optimization algorithms is not straightforward. The only theoretical result in this regard, to the best of our knowledge, is \cite{NEURIPS2021_74e1ed8b}, which incorporates the one-peer exponential graphs into the Decentralized Momentum Stochastic Gradient (DmSGD) method. The authors of \cite{NEURIPS2021_74e1ed8b} show that the convergence rate of DmSGD using one-peer exponential graphs is the same as DmSGD using static exponential graphs. Yet, the final convergence rate derived in \cite{NEURIPS2021_74e1ed8b} fails to consider the stepsize condition imposed by their analysis. Specifically, to derive the final convergence rate in~\cite[Equation~(107)]{NEURIPS2021_74e1ed8b}, the authors select a stepsize that may be incompatible with some previously imposed stepsize conditions (\ie, \cite[Equation (105)]{NEURIPS2021_74e1ed8b}). Moreover, we note that the analysis of DSGD provided in~\cite{koloskova2020unified} can be adapted to encapsulate finite-time consensus graphs while DSGD is a special case of DmSGD (with the momentum parameter $\beta = 0$).
\begin{table*}[t]
\caption{Classes of graph sequences that satisfy the finite-time consensus property.}
\label{table:fin-con-topo}
\begin{center}
    \vspace{-10pt}
    \begin{adjustbox}{max width=\textwidth}
    \begin{threeparttable}
    \begin{tabular}{c c c c c} \toprule
        Topology & Orientation & Size $n$ & Maximum degree & 
        \begin{tabular}[c]{@{}c@{}} Num.\ of iterations ($\tau$) \\ for finite-time consensus \end{tabular} \\ \toprule
        \multicolumn{1}{l}{de Bruijn~\cite{delvenne2009optimal}} & Directed & Power of $p \in \natint_{\geq 2}$ & $p-1$ & $\log_p (n)$ \\ \midrule
        \multicolumn{1}{l}{One-peer exponential~\cite{NEURIPS2021_74e1ed8b}} & Directed & Power of $2$~$^\text{a}$ & $1$ & $\log_2(n)$ \\ \midrule
        \multicolumn{1}{l}{One-peer hyper-cube~\cite{Shi_2016}} & Undirected & Power of $2$ & $1$ & $\log_2(n)$ \\ \midrule
        \multicolumn{1}{l}{$p$-Peer hyper-cuboid} & Undirected & Any $n \in \natint$ & flexible$^\text{b}$ & flexible$^\text{b}$ \\ \bottomrule
    \end{tabular}
    \begin{tablenotes}
        \item[a] When $n \neq 2^\tau$, one-peer exponential graphs are still well defined but do not have finite-time consensus property.
        
        \item[b] For $p$-peer hyper-cuboids, the maximum degree and the number of iterations for finite-time consensus depend on a factorization of the node size $n$ and exhibits a trade-off. For example, when $n=20$ is factorized as $n=20=2 \times 2 \times 5$, the number of iterations for finite-time consensus is $\tau=3$, the number of factors, while the maximum degree in all the three graphs is $4$, the largest factor minus one. Another possible factorization is $n=20=2 \times 10$, in which case $\tau=2$ and the maximum degree (in the two graphs) is $9$. Also note that when~$n$ is a prime number, the $p$-peer hyper-cuboids reduce to the fully-connected graph.
    \end{tablenotes}    
    \end{threeparttable}
    \end{adjustbox}
\end{center}
\end{table*}

\section{Finite-time consensus} \label{sec:fin-time-cons}

In this section, we formally define the finite-time consensus property, and establish this property for several sequences of graphs. For one-peer exponential graphs (with $n=2^\tau$) presented in \cite{NEURIPS2021_74e1ed8b}, we present an alternative proof based on the circulant property. Moreover, we present $p$-peer hyper-cuboids of an arbitrary node size, and establish their finite-time consensus property. Finally, we show that in certain cases, the presented $p$-peer hyper-cuboids are permutation equivalent to the de Bruijn graphs.

\paragraph{Notation}
Only in this section, we modify the convention for indexing a matrix $W \in \reals^{n \times n}$. In this section, we start the indexing at $0$, \textit{i.e.}, the entries of an $n \times n$ matrix $W$ are
\[
    W = [w_{ij}], \quad \text{for } i, j = 0,1,\ldots, n-1.
\]
In the rest of the paper, we follow the convention and start the indexing at $1$, \textit{i.e.},
\[
    W = [w_{ij}], \quad \text{for } i, j = 1,2,\ldots, n.
\]

\subsection{Definition}

We formally present the conditions for a sequence of graphs (or topologies) to have the finite-time consensus (or exact averaging) property.
\begin{definition}
    \label{def:fin-time-cons}
    The sequence of graphs $\cG^{(l)} = (\cV, W^{(l)}, \cE^{(l)})$, $l = 0, \ldots, \tau-1$, has the finite-time consensus property with parameter $\tau \in \natint_{\geq 1}$ if and only if the weight matrices $\{W^{(l)}\}_{l=0}^{\tau-1}$ are doubly stochastic and satisfy
    \begin{equation} \label{eq:fin-time-cons-cond}
        W^{(\tau - 1)}W^{(\tau - 2)} \cdots W^{(1)} W^{(0)} = \tfrac{1}{n} \ones \ones\tran.
    \end{equation}
\end{definition}
\par
For a sequence of doubly stochastic matrices $\{W^{(l)}\}_{l=0}^{\tau-1}$, \eqref{eq:fin-time-cons-cond} is equivalent to
\begin{equation} \label{eq:fin-time-cons-residual}
    \big(W^{(\tau-1)} - \tfrac{1}{n} \ones \ones\tran \big) \big(W^{(\tau-2)} - \tfrac{1}{n} \ones \ones\tran \big) \cdots \big(W^{(1)} - \tfrac{1}{n} \ones \ones\tran) \big(W^{(0)} - \tfrac{1}{n} \ones \ones\tran \big) = 0.
\end{equation}
The parameter $\tau$ in \cref{def:fin-time-cons} is not arbitrary and might depend on the matrix size $n$ and graph structure. For the examples provided in \Cref{sec:exp-gra,sec:hyper-cube,sec:hyper-cuboid}, the parameter~$\tau$ will be clear from the context. In \cref{def:fin-time-cons}, each matrix $W^{(l)}$ in the sequence is required to be doubly stochastic but needs not be symmetric or connected. The potential disconnectivity might be a desirable property in the context of decentralized optimization, because such graphs tend to be sparser and using them in decentralized algorithms helps reduce the communication overhead at each iteration. Nevertheless, when considered jointly, a sequence of graphs satisfying \cref{def:fin-time-cons} exhibits the same connectivity properties as the fully connected graph. 
\par
More discussion on the existing graph sequences with finite-time consensus can be found in \Cref{sec:related}. In the remainder of this section, we detail the one-peer exponential graphs and the $p$-peer hyper-cuboids. For the well-studied de Bruijn graphs \cite{Bruijn1946ACP,delvenne2009optimal}, we establish its connection to the $p$-peer hyper-cuboids.

\subsection{One-peer exponential graphs} \label{sec:exp-gra}

We present the weight matrices of one-peer exponential graphs (which was developed in \cite{NEURIPS2021_74e1ed8b}) and list several properties of our interest. In particular, the weight matrices representing one-peer exponential graphs are asymmetric, doubly stochastic, sparse (potentially disconnected), and when $n$ is a power of 2, satisfy~\eqref{eq:fin-time-cons-cond}. As a byproduct, we develop an alternative proof for the finite-time consensus property of one-peer exponential graphs (when $n=2^\tau$).
\par
For a given matrix size $n \in \natint_{\geq 2}$, let $\tau = \lceil \log_2(n) \rceil$. Then, the weight matrices representing the one-peer exponential graphs $\{\cG^{(l)}\}_{l \in \natint}$ are defined as
\begin{equation} \label{def:exp-mat}
    w^{(l)}_{ij} = 
    \begin{cases}
        \frac{1}{2} & \text{if } \text{mod}(j-i, n) = 2^{\text{mod}(l, \tau)} \\
        \frac{1}{2} & \text{if } i = j \\
        0 & \text{otherwise,}
    \end{cases}
\end{equation}
and \cref{fig:expo-8} shows the three one-peer exponential graphs when $n=8$ and $\tau=3$. By definition, given $n \in \natint_{\geq 2}$, there exists a number of $\tau = \lceil \log_2 (n) \rceil$ distinct one-peer exponential graphs. All of them are asymmetric, doubly stochastic, circulant, and some of them are \textit{not} strongly connected. When one-peer exponential graphs are used as network topology in decentralized algorithms, each agent only communicates with one neighbor at each iteration. Thus, the total communication cost per iteration is $\Theta (1)$. Finally, the crucial property that makes one-peer exponential graphs useful in decentralized optimization is presented as follows.
\begin{proposition} \label{lem:exp-fin-time-cons}
    Given $n \in \natint_{\geq 2}$, let $\tau = \lceil \log_2 (n) \rceil$, and let $\{W^{(l)}\}_{l \in \natint} \subset \reals^{n \times n}$ be the weight matrices defined in~\eqref{def:exp-mat}. Each matrix $W^{(l)}$ is circulant and doubly stochastic, i.e., $W^{(l)} \ones = \ones$ and $\ones\tran W^{(l)} = \ones\tran$. In addition, if $n$ is a power of 2 (\ie, $n = 2^\tau$) and the index sequence $\{l_i\}_{i=0}^{\tau-1}$ satisfies $\{\mathrm{mod} (l_0, \tau), \dots, \mathrm{mod}(l_{\tau-1}, \tau)\} = \{0, \dots, \tau - 1\}$, then the matrices $\{W^{(l_i)}\}_{i=0}^{\tau-1}$ satisfy the finite-time consensus property in \cref{def:fin-time-cons}.
\end{proposition}
From \cref{lem:exp-fin-time-cons}, given a sequence of one-peer exponential graphs with $n=2^\tau$, any permutation of this sequence will still satisfy the finite-time consensus property~\eqref{eq:fin-time-cons-cond}. The properties in \cref{def:fin-time-cons} are the minimum requirements needed for algorithm analysis in \Cref{sec:analysis}. Additional properties, such as the circulant property stated in \cref{lem:exp-fin-time-cons}, are desirable but unnecessary in algorithm design.

\begin{figure}
\centering
    \begin{subfigure}[b]{0.3\textwidth}
    \centering
    \begin{tikzpicture}[->,>=Stealth,shorten >=1pt,auto,node distance=1.0cm,
                        thick,main node/.style={circle,draw,font=\sffamily\small\bfseries,minimum size=0.3cm}]
        \begin{scope}[shift={(0,0)}]
            \foreach \i in {0,...,7} {
                \node[main node] (\i) at ({360/8 * (\i)}:1.5cm) {$\i$};
            }
            \foreach \i in {0,...,6} {
                \draw[->] (\i) to[bend right=20] (\the\numexpr\i+1\relax);
            }
            \draw[->] (7) to[bend right=20] (0);
        \end{scope}
    \end{tikzpicture}
    \caption{$\cG^{(0)}$.}
    \end{subfigure}
    \hfill
    \begin{subfigure}[b]{0.3\textwidth}
    \centering
    \begin{tikzpicture}[->,>=Stealth,shorten >=1pt,auto,node distance=1.0cm,
                        thick,main node/.style={circle,draw,font=\sffamily\small\bfseries,minimum size=0.3cm}]
        \begin{scope}[shift={(10cm,0)}]
            \foreach \i in {0,...,7} {
                \node[main node] (\i) at ({360/8 * (\i)}:1.5cm) {$\i$};
            }
            \foreach \i/\j in {0/2, 2/4, 4/6, 6/0} {
                \draw[->,in=135,out=145] (\i) to[bend right=-10] (\j);
            }
            \foreach \i/\j in {1/3, 3/5, 5/7, 7/1} {
                \draw[->,in=135,out=45] (\i) to[bend right=-10] (\j);
            }
        \end{scope}
    \end{tikzpicture}
    \caption{$\cG^{(1)}$.}
    \end{subfigure}
    \hfill
    \begin{subfigure}[b]{0.3\textwidth}
    \centering
    \begin{tikzpicture}[->,>=Stealth,shorten >=1pt,auto,node distance=1.0cm,
                        thick,main node/.style={circle,draw,font=\sffamily\small\bfseries,minimum size=0.3cm}]
        \begin{scope}[shift={(5cm,0)}]
            \foreach \i in {0,...,7} {
                \node[main node] (\i) at ({360/8 * (\i)}:1.5cm) {$\i$};
            }
            \foreach \i in {0,...,3} {
                \draw[<->] (\i) to (\the\numexpr\i+4\relax);
            }
        \end{scope}
    \end{tikzpicture}
    \caption{$\cG^{(2)}$.}
    \end{subfigure}
\caption{The three one-peer exponential graphs $\{\cG^{(l)}\}_{l=0}^2$ with $n=8$ and $\tau = 3$. Note that all nodes have self-loops, though not explicitly shown in the figure.}
\label{fig:expo-8}
\end{figure}
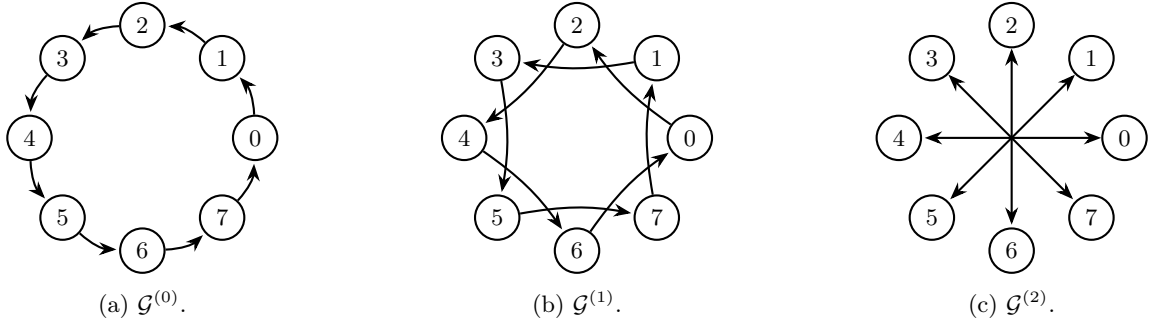

Our proof of \cref{lem:exp-fin-time-cons} relies on some basic properties of \textit{circulant matrices}, which are summarized in the following lemma and can be found in, \eg, in \cite[\S4.7.7]{horn13}.
\begin{lemma} \label{lem:circulant-matrix}
    The $n \times n$ circulant matrix associated with the $n$-vector \\ $c=(c_0,c_1,\ldots,c_{n-1})$ is defined by
    \begin{equation} \label{eq:circ-mat}
        C = \mathrm{Circ} (c_0,c_1,\ldots,c_{n-1}) \define \begin{bmatrix}
            c_0 & c_{n-1} & \cdots & c_2 & c_1 \\ 
            c_1 & c_0 & c_{n-1} & & c_2 \\ 
            \vdots & c_1 & c_0 & \ddots & \vdots \\ 
            c_{n-2} & & \ddots & \ddots & c_{n-1} \\ 
            c_{n-1} & c_{n-2} & \ddots & c_1 & c_0
        \end{bmatrix}.
    \end{equation}
    Let $\jm$ be the imaginary number ($\jm^2 = -1$) and let $\omega = \exp \big(\tfrac{2 \pi \jm}{n} \big)$ be a primitive $n$-th root of unity. Denote by $\mathsf{F}$ by the $n \times n$ discrete Fourier transform (DFT) matrix ($\mathsf{F}_{ij} = \tfrac{1}{\sqrt{n}} \omega^{ij}$), and denote by $\mathsf{F}\conj$ the Hermitian (conjugate transpose) of $\mathsf{F}$. Then, the eigenvalue decomposition of $C$~\eqref{eq:circ-mat} is given by
    \[
        C = \big(\tfrac{1}{\sqrt n} \mathsf{F}\big) \cdot \big(\mathrm{diag} (\mathsf{F} c) \big) \cdot \big(\tfrac{1}{\sqrt n} \mathsf{F} \conj \big),
    \]
    where the $\mathrm{diag}$ operator transforms an $n$-vector into an $n \times n$ diagonal matrix. Moreover, the eigenvalues of $C$~\eqref{eq:circ-mat} are given by
    \begin{equation} \label{eq:circ-eig}
        \lambda_k = c_0 + c_1 \omega^k + c_2 \omega^{2k} + \cdots + c_{n-1} \omega^{(n-1)k}, \quad k = 0,1,\ldots,n-1.
    \end{equation}
\end{lemma}

Now we present the proof for \cref{lem:exp-fin-time-cons}. The original proof in \cite{NEURIPS2021_74e1ed8b} leverages the binary representation of the integer $n$, but our proof is inspired by the circulant property of one-peer exponential graphs. (Note that the circulant property has been exploited in \cite{NEURIPS2021_74e1ed8b} to compute the eigenvalues of the \textit{static} exponential graph but was not used to establish finite-time consensus.)
\begin{proof}[Proof of \cref{lem:exp-fin-time-cons}]
    Note that the mixing matrices $\{W^{(l)}\}_{l=0}^{\tau-1}$ of one-peer exponential graphs defined in~\eqref{def:exp-mat} are circulant. Thus, we have
    \[
        W^{(\tau-1)} \cdots W^{(1)} W^{(0)} = \big(\tfrac{1}{\sqrt{n}} \mathsf{F} \big) \cdot \big(\Lambda^{({\tau-1})} \cdots \Lambda^{(1)} \Lambda^{(0)} \big) \cdot \big(\tfrac{1}{\sqrt{n}} \mathsf{F}\conj \big),
    \]
    where $\Lambda^{(l)} = \mathrm{diag} (\mathsf{F} c^{(l)})$ and $c^{(l)}$ is the first column of $W^{(l)}$, for $l=0,1,\ldots,\tau-1$. For simplicity, we denote $\Lambda \define \Lambda^{({\tau-1})} \cdots \Lambda^{(1)} \Lambda^{(0)}$.
    \par
    For each $l = 0,1,\ldots,\tau-1$, it follows from \eqref{eq:circ-eig} in \cref{lem:circulant-matrix} that the first element in $\mathsf{F} c^{(l)}$ is 1. This implies that the first element in $\Lambda$ is also 1. Similarly, the $k$-th diagonal element in~$\Lambda$, denoted by $\Lambda[k]$, can be found by expanding the definition $\Lambda^{(l)} = \mathrm{diag} (\mathsf{F} c^{(l)})$:
    \begin{align*}
        \Lambda [k] &= (\Lambda^{(\tau-1)} [k]) \cdots (\Lambda^{(2)}[k])(\Lambda^{(1)} [k]) (\Lambda^{(0)} [k])  \\
        &= \frac{1}{2^\tau} \left( (1 + \omega^{k (n-2^{\tau - 1})})  \cdots (1+\omega^{k(n-4)}) (1+\omega^{k(n-2)}) (1+\omega^{k(n-1)}) \right) \\
        &= \frac{1}{2^\tau} \left( (1 + \omega^{-2^{\tau - 1}k}) \cdots (1+\omega^{-4k})(1+\omega^{-2k})  (1+\omega^{-k}) \right) \\
        & = \frac{1}{2^\tau} \sum_{q = 0}^{n-1} \omega^{-k q}
        = \frac{1}{2^\tau} \left(\frac{1-\omega^{-k n}}{1-\omega^{-k}} \right)
        = 0,
    \end{align*}
    where in the first equation we use the assumption that $n = 2^\tau$. Hence, we have
    \begin{align*}
        W^{({\tau-1})} \cdots W^{(1)} W^{(0)} &= \big(\tfrac{1}{\sqrt{n}} \mathsf{F}\big) \Lambda^{({\tau-1})} \cdots \Lambda^{(1)} \Lambda^{(0)} \big(\tfrac{1}{\sqrt{n}} \mathsf{F}\conj \big) \\
        &= \big(\tfrac{1}{\sqrt{n}} \mathsf{F}\big) \diag (1,0,0,\dots,0) \big(\tfrac{1}{\sqrt{n}} \mathsf{F}\conj \big) \\
        &= \tfrac{1}{n} \ones \ones\tran,
    \end{align*}
    which establishes the desirable identity~\eqref{eq:fin-time-cons-cond}. To show~\eqref{eq:fin-time-cons-residual}, we first observe that
    \[
        \big(W^{({l+1})} - \tfrac{1}{n} \ones \ones\tran \big) \big(W^{(l)} - \tfrac{1}{n} \ones \ones\tran \big) = W^{({l+1})} W^{(l)} - \tfrac{1}{n} \ones \ones\tran, \ \ \text{for all} \ l=0,\ldots,\tau-2.
    \]
    This equality holds due to the doubly stochastic property of $W^{(l)}$. We can then repeat this process to obtain
    \[
        \big(W^{({\tau-1})} - \tfrac{1}{n} \ones \ones\tran \big) \cdots \big(W^{(0)} - \tfrac{1}{n} \ones \ones\tran \big) = W^{({\tau-1})} \cdots W^{(0)} - \tfrac{1}{n} \ones \ones\tran = 0,
    \]
    where in the last step we use~\eqref{eq:fin-time-cons-cond}. Combining this result with the fact that one-peer exponential graphs are periodic, and all instances of one-peer exponential graphs are commutative, we obtain that any permutation of any sequence of $W^{(l)}$ of length larger than $\tau$ has the finite-time consensus property as well.
\end{proof}

\subsection{One-peer hyper-cubes} \label{sec:hyper-cube}

Another example of graph sequences that satisfy \cref{def:fin-time-cons} is one-peer hyper-cubes~\cite{Shi_2016} (with $n$ a power of $2$). Hyper-cubes have been extensively studied in theoretical computer science (see, \eg, \cite{harary1988hypercube} for a survey). Still, the specialization to \textit{one-peer} hyper-cubes with finite-time consensus was first, to the best of our knowledge, discussed in~\cite{Shi_2016}. The formal definition of one-peer hyper-cubes is presented here to keep our work self-contained and, more importantly, to motivate the extension to $p$-peer hyper-cuboids for any $n \in \natint_{\geq 2}$ in \Cref{sec:hyper-cuboid}.

Let $n = 2^\tau$ with $\tau \in \natint_{\geq 1}$, the weight matrices $\{W^{(l)}\}_{l \in \natint} \subset \reals^{n \times n}$ representing the one-peer hyper-cubes $\{\cG^{(l)}\}_{l \in \natint}$ are defined by
\begin{equation} \label{def:hypercube-mat}
    w^{(l)}_{ij} = \begin{cases}
        \frac{1}{2} & \text{if } (i \wedge j) =2^{{\rm mod}(l, \tau)} \\
        \frac{1}{2} & \text{if } i = j \\
        0 & \text{otherwise},
    \end{cases}
\end{equation}
where the notation $i \wedge j$ represents the bit-wise XOR operation between integers $i$ and~$j$.
The difference in the definition of one-peer hyper-cubes~\eqref{def:hypercube-mat} and that of one-peer exponential graphs~\eqref{def:exp-mat} is minor yet critical: The operation $(i \wedge j)$ is used in~\eqref{def:hypercube-mat} while $\mathrm{mod} (j-i,n)$ in~\eqref{def:exp-mat}. Since $(i \wedge j) = (j \wedge i)$, it immediately follows that the mixing matrices of one-peer hyper-cubes are symmetric. We represent the integer $i$ in its binary form {${(i_{\tau-1} i_{\tau-2} \cdots i_0)}_2$}. Then, the first if-condition in~\eqref{def:hypercube-mat} can be rewritten as
\[ \vspace{-3pt}
    (i_{\tau-1} i_{\tau-2} \cdots i_0)_2 \wedge (j_{\tau-1} j_{\tau-2} \cdots j_0)_2 = (0 \cdots 0 \,1\underbrace{\,0\,\cdots\,0}_{\mathrm{mod}(l, \tau)})_2; \vspace{-3pt}
\]
that is, only the $({\mathrm{mod}(l, \tau)}+1)$-th digit in $i$'s and $j$'s binary representation is different, and the rest of the digits are the same. To construct one-peer hyper-cubes, we first index the vertices as $\tau$-digit binary numbers, and then an edge is created between two distinct vertices if their binary representations differ by a single digit. Finally, the finite-time consensus property proof for one-peer hyper-cubes is postponed to \Cref{sec:hyper-cuboid}, since one-peer hyper-cubes will be covered as a special case of the $p$-peer hyper-cuboids.

\subsection{\texorpdfstring{$p$}{p}-Peer hyper-cuboids} \label{sec:hyper-cuboid}

\begin{figure}
\centering
\begin{subfigure}[b]{0.3\textwidth}
\centering
\resizebox{1.0\columnwidth}{!}{
    \begin{tikzpicture}[>=Stealth,shorten >=1pt,auto,node distance=0.9cm,
                    thick,main node/.style={circle,draw,font=\small\bfseries,minimum size=0.7cm},scale=1.2]
    \begin{scope}[shift={(0,0)}]
        \foreach \x/\y/\z/\label in {0/0/0/9, 1.5/0/0/10, 1.5/1.5/0/7, 0/1.5/0/6, 0/0/2/3, 1.5/0/2/4, 1.5/1.5/2/1, 0/1.5/2/0, 3/0/0/11, 3/1.5/0/8, 3/0/2/5, 3/1.5/2/2}
        \node[main node] (\label) at (\x, \y, \z) {$\label$};
  
        \draw[gray!20] (0) -- (1) -- (2) -- (5) -- (4) -- (3) -- (0) -- (6) -- (7) -- (8) -- (11) to (5);
        \draw[gray!20] (1) to (7) to (10) to (4) to (1);
        \draw[gray!20] (3) to (9) to (6);
        \draw[gray!20] (11) to (10) to (9);
        \draw[gray!20] (2) to (8);

        \draw[<->] (0) to (1);
        \draw[<->] (1) to (2);
        \draw[<->] (6) to (7);
        \draw[<->] (7) to (8);

        \draw[<->] (9)  to (10);
        \draw[<->] (10) to (11);
        \draw[<->] (3)  to (4);
        \draw[<->] (4)  to (5);

        \draw[<->] (0) to[bend right=-20] (2);
        \draw[<->] (6) to[bend right=-20] (8);
        \draw[<->] (9) to[bend right=20]  (11);
        \draw[<->] (3) to[bend right=20]  (5);
    \end{scope}
    \end{tikzpicture}
}
\caption{$\cG^{(0)}$.}
\end{subfigure} 
\hfill
\begin{subfigure}[b]{0.3\textwidth}
\centering
\resizebox{1.0\columnwidth}{!}{
    \begin{tikzpicture}[>=Stealth,shorten >=1pt,auto,node distance=0.9cm,
                    thick,main node/.style={circle,draw,font=\small\bfseries,minimum size=0.7cm},scale=1.2]
    \begin{scope}[shift={(0,0)}]
        \foreach \x/\y/\z/\label in {0/0/0/9, 1.5/0/0/10, 1.5/1.5/0/7, 0/1.5/0/6, 0/0/2/3, 1.5/0/2/4, 1.5/1.5/2/1, 0/1.5/2/0, 3/0/0/11, 3/1.5/0/8, 3/0/2/5, 3/1.5/2/2}
        \node[main node] (\label) at (\x, \y, \z) {$\label$};
  
        \draw[gray!20] (0) -- (1) -- (2) -- (5) -- (4) -- (3) -- (0) -- (6) -- (7) -- (8) -- (11) to (5);
        \draw[gray!20] (1) to (7) to (10) to (4) to (1);
        \draw[gray!20] (3) to (9) to (6);
        \draw[gray!20] (11) to (10) to (9);
        \draw[gray!20] (2) to (8);

        \draw[<->] (0) to (3);
        \draw[<->] (1) to (4);
        \draw[<->] (6) to (9);
        \draw[<->] (7) to (10);

        \draw[<->] (8) to (11);
        \draw[<->] (2) to (5);

        \vphantom{
        \draw[<->] (0) to[bend right=-20] (2);
        \draw[<->] (6) to[bend right=-20] (8);
        \draw[<->] (9) to[bend right=20]  (11);
        \draw[<->] (3) to[bend right=20]  (5);
        }
    \end{scope}
    \end{tikzpicture}
}
\caption{$\cG^{(1)}$.}
\label{fig:hyper-cuboid-b}
\end{subfigure}
\hfill
\begin{subfigure}[b]{0.3\textwidth}
\centering
\resizebox{1.0\columnwidth}{!}{
\begin{tikzpicture}[>=Stealth,shorten >=1pt,auto,node distance=0.9cm,
                    thick,main node/.style={circle,draw,font=\small\bfseries,minimum size=0.7cm},scale=1.2]

    \begin{scope}[shift={(0,0)}]
        \foreach \x/\y/\z/\label in {0/0/0/9, 1.5/0/0/10, 1.5/1.5/0/7, 0/1.5/0/6, 0/0/2/3, 1.5/0/2/4, 1.5/1.5/2/1, 0/1.5/2/0, 3/0/0/11, 3/1.5/0/8, 3/0/2/5, 3/1.5/2/2}
        \node[main node] (\label) at (\x, \y, \z) {$\label$};
  
        \draw[gray!20] (0) -- (1) -- (2) -- (5) -- (4) -- (3) -- (0) -- (6) -- (7) -- (8) -- (11) to (5);
        \draw[gray!20] (1) to (7) to (10) to (4) to (1);
        \draw[gray!20] (3) to (9) to (6);
        \draw[gray!20] (11) to (10) to (9);
        \draw[gray!20] (2) to (8);

        \draw[<->] (0)  to (6);
        \draw[<->] (10) to (4);
        \draw[<->] (3)  to (9);
        \draw[<->] (7)  to (1);

        \draw[<->] (8)  to (2);
        \draw[<->] (11) to (5);

        \vphantom{
        \draw[<->] (0) to[bend right=-20] (2);
        \draw[<->] (6) to[bend right=-20] (8);
        \draw[<->] (9) to[bend right=20]  (11);
        \draw[<->] (3) to[bend right=20]  (5);
        }
    \end{scope}
\end{tikzpicture}
}
\caption{$\cG^{(2)}$.}
\end{subfigure}
\caption{The three $2$-peer hyper-cuboids $\{\cG^{(l)}\}_{l=0}^2$ with $n=12$, $(p_2,p_1,p_0) = (2,2,3)$, and $\tau=3$. Note that all nodes have self-loops, though not explicitly shown here.}
\label{fig:hyper-cuboid}
\end{figure}
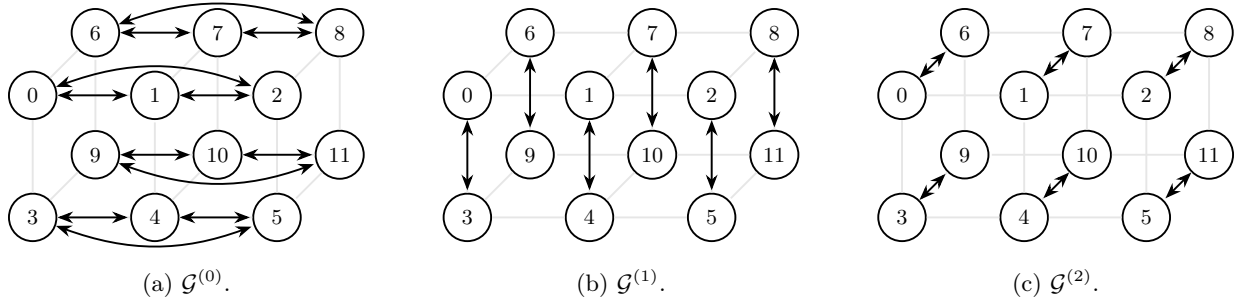

Both one-peer hyper-cubes and one-peer exponential graphs enjoy the finite-time consensus property when the number of agents is a power of 2. One natural extension of hyper-cubes to admit an arbitrary number of agents is the \textit{hyper-cuboid}, which has many different names (\eg, hyper-box, orthotope) and has been well studied in, \eg, \cite{coxeter1973regular}. So, borrowing the idea behind the one-peer hyper-cubes, we present a family of sparse graphs that achieves finite-time consensus for any integer $n \in \natint_{\geq 2}$. 
\par
Recall that one-peer hyper-cubes~\eqref{def:hypercube-mat} are defined via the binary representation of integers. Then, the extension to arbitrary matrix size $n$ relies on a \textit{multi-base representation} of integers \cite{krenn2015multi}. To be specific, the $(p_{\tau-1}, p_{\tau-2}, \ldots, p_0)$-based representation of an integer is an element in the group $\natint_{p_{\tau-1}} \times \natint_{p_{\tau-2}} \times \cdots \times \natint_{p_0}$, where $\natint_{p_k}$ is the group of nonnegative integers modulo $p_k \in \natint_{\geq 2}$. Any natural integer smaller than $p_{\tau-1} \times p_{\tau-2} \times \cdots \times p_0$ finds a one-to-one mapping in this group. For example, a $(2,2,\ldots,2)$-based representation is equivalent to the binary representation of an integer. A more informative example is the $(2,3)$-based representation. In this case, we can map the integers in $\{0,1,\ldots,5\}$ according to the following rule:
\begin{gather*}
    0 \to \{0\}_2 \times \{0\}_3, \quad 1 \to \{0\}_2 \times \{1\}_3, \quad 2 \to \{0\}_2 \times \{2\}_3, \\
    3 \to \{1\}_2 \times \{0\}_3, \quad 4 \to \{1\}_2 \times \{1\}_3, \quad 5 \to  \{1\}_2 \times \{2\}_3.
\end{gather*}
To shorten the notation, we overload binary representation and denote the \\ $(p_{\tau-1}, p_{\tau-2}, \ldots, p_0)$-based representation of $i \in \natint$ as $({i_{p_{\tau-1}} \cdots i_{p_1} i_{p_0}})_{p_{\tau-1}, \ldots, p_1, p_0}$, so that we can also re-write $\{a\}_{p_1} \times \{b\}_{p_0}$ as $(a,b)_{p_1,p_0}$.
\par
Now, we are ready to construct $p$-peer hyper-cuboids with $n$ agents. Suppose the prime factorization of~$n$ is given by $n = p_{\tau-1} \cdots p_1 p_0$, where all the $p_j$ are prime numbers, and define the $\tau$-vector $p=(p_0,p_1,\ldots,p_{\tau-1}) \in \natint^\tau$. (It is possible that $p_i = p_j$ for $i \neq j$, and the order of $\{p_j\}$ does not matter. Also, prime factorization is considered here only for simplicity. Any factorization of $n$ would build a sequence of $p$-peer hyper-cuboids; see also the footnote of \cref{table:fin-con-topo}.) Then, the weight matrices of $p$-peer hyper-cuboids are defined by 
\begin{equation} \label{def:hypercuboid-mat}
    w^{(l)}_{ij} = \begin{cases}
        \frac{1}{p_{\mathrm{mod}(l, \tau)}} & \text{if} \ (i \wedge_{p_{\tau-1}, \ldots, p_1, p_0} j) = (0, \cdots, 0, 1, \underbrace{0,\cdots,0}_{\mathrm{mod}(l, \tau)})_{p_{\tau-1}, \ldots, p_1, p_0} \\
        \frac{1}{p_{\mathrm{mod}(l, \tau)}} & \text{if} \ i = j \\
        0 & \text{otherwise},
    \end{cases}
\end{equation}
for $l=0,1,\ldots,\tau-1$, where $i \wedge_{p_{\tau-1}, \ldots, p_1, p_0} j$ denotes the bit-wise XOR operation between the $(p_{\tau-1}, \ldots, p_1, p_0)$-based representations of $i$ and $j$; that is, if the $p_k \in \natint_{p_k}$ element of $i$'s multi-base representations is the same as that of $j$, then return $\{0\}_{p_k}$, and otherwise return $\{1\}_{p_k}$. \cref{fig:hyper-cuboid} shows all three distinct $2$-peer hyper-cuboids with $12$ agents. In this example, $n=12$, $(p_2, p_1, p_0) = (2,2,3)$, and $\tau=3$. To illustrate the definition~\eqref{def:hypercube-mat}, take the edge $(i,j) = (8,11)$ in $\cG^{(1)}$ as an example; see \cref{fig:hyper-cuboid-b}. The two integers $i=8$ and $j=11$ are mapped in the $(2,2,3)$-based representation as
\[
    8 \to \{1\}_2 \times \{0\}_2 \times \{2\}_3, \qquad 11 \to \{1\}_2 \times \{1\}_2 \times \{2\}_3.
\]
These two representations differ only at the second sub-group $\natint_{p_1} = \natint_2$, and thus when $l=1$, agents $i=8$ and $j=11$ are connected with weight $w^{(l)}_{ij} = w^{(1)}_{8,11} = \tfrac{1}{p_1} = \tfrac{1}{2}$.

The definition of $p$-peer hyper-cuboids~\eqref{def:hypercuboid-mat} is clear as an extension from binary numbers to multi-base integer representations. Yet, the definition~\eqref{def:hypercuboid-mat} is less intuitive when we try to establish the properties of $p$-peer hyper-cuboids. It turns out that the weight matrix $W^{(l)}$ of $p$-peer hyper-cuboids defined in \eqref{def:hypercuboid-mat} also has an elegant representation in terms of Kronecker products:
\begin{equation} \label{def:hypercuboid-mat-kron}
    W^{(l)} = W^{(l)}(p_{\tau-1}) \otimes \ldots \otimes W^{(l)}(p_1)  \otimes W^{(l)}(p_{0}), 
\end{equation}
where each $p_r \times p_r$ matrix $W^{(l)}(p_r)$ is defined by
\begin{equation} \label{eq:hypercuboid-prf-1}
    W^{(l)} (p_r) = \begin{cases}
        I_{p_r} &\text{if } \mathrm{mod} (l, \tau) \neq r \\
        \tfrac{1}{p_r} \ones \ones\tran \qquad &\text{if} \ \mathrm{mod} (l, \tau) = r.
    \end{cases}  
\end{equation}
The equivalence between~\eqref{def:hypercuboid-mat} and~\eqref{def:hypercuboid-mat-kron} can be established as follows.
\begin{subequations}
\begin{align}
    W^{(l)} &= \sum_{i=0}^{n-1} \sum_{j=0}^{n-1} w^{(l)}_{ij} e_{i} e_{j}\tran \nonumber \\
    &= \sum_{i_{p_{\tau-1}}, \ldots, i_{p_0}} \hspace{-.4em}{\cdots} \hspace{-.4em} \sum_{j_{p_{\tau-1}}, \ldots, j_{p_0}} w^{(l)}_{ij} (\ehat_{i_{p_{\tau-1}}} \otimes \hspace{-.25em} {\cdots} \hspace{-.25em} \otimes \ehat_{i_{p_{0}}})(\ehat_{j_{p_{\tau-1}}} \otimes \hspace{-.25em} {\cdots} \hspace{-.25em} \otimes \ehat_{j_{p_{0}}})\tran \nonumber \\
    &= \sum_{i_{p_{\tau-1}}, \ldots, i_{p_0}} \sum_{j_{p_r} \colon r = \mathrm{mod} (l, \tau)} \tfrac{1}{p_{r}} (\ehat_{i_{p_{\tau-1}}} \ehat_{i_{p_{\tau-1}}}\tran) \otimes \hspace{-.25em} {\cdots} \hspace{-.25em} \otimes (\ehat_{i_{p_r}} \ehat_{j_{p_r}}\tran) \otimes \hspace{-.25em} {\cdots} \hspace{-.25em} \otimes (\ehat_{i_{p_{0}}} \ehat_{i_{p_{0}}}\tran) \label{eq:hypercuboid-prf-2} \\
    &= \Big(\sum\limits_{i_{p_{\tau-1}}} \ehat_{i_{p_{\tau-1}}} \ehat_{i_{p_{\tau-1}}}\tran \Big) \otimes \cdots \otimes \Big(\sum\limits_{i_{p_r}} \sum\limits_{j_{p_r}} \tfrac{1}{p_{r}} \ehat_{i_{p_r}} \ehat_{j_{p_r}}\tran \Big) \otimes \cdots \otimes \Big(\sum\limits_{i_{p_{0}}} \ehat_{i_{p_{0}}} \ehat_{i_{p_{0}}}\tran \Big) \label{eq:hypercuboid-prf-3} \\
    &= W^{(l)}(p_{\tau-1}) \otimes \cdots \otimes W^{(l)}(p_1) \otimes W^{(l)}(p_{0}). \label{eq:hypercuboid-prf-4}
\end{align}
\end{subequations}
Here, the notation $e_i$ is a base unit vector of length $n$, \ie, all entries are 0 except that the $i$-th entry is~$1$. (Recall that the index starts with $0$.) The notation $\hat e_{i_{p_r}}$ is a base unit vector of length $p_r$, where all entries are $0$ except for the \mbox{$(i_{p_r})$-th} entry, and recall $p_r$ is a prime factor of $n$. Step~\eqref{eq:hypercuboid-prf-2} removes the terms that are $0$ using definition~\eqref{def:hypercuboid-mat}, and then we apply the transpose and mixed-product properties of the Kronecker product. Step~\eqref{eq:hypercuboid-prf-3} distributes each summation into the corresponding Kronecker product, and~\eqref{eq:hypercuboid-prf-4} uses the definition of $W^{(l)} (p_r)$ in~\eqref{eq:hypercuboid-prf-1}.
\par
Now, we establish the finite-time consensus property of $p$-peer hyper-cuboids.
\begin{proposition}
    Given $n \in \natint_{\geq 2}$, let $\{W^{(l)}\}_{l \in \natint} \subset \reals^{n \times n}$ be the weight matrices defined in~\eqref{def:hypercuboid-mat}. Each matrix $W^{(l)}$ is symmetric and doubly stochastic. For an index sequence $\{l_i\}_{i=0}^{\tau-1}$ with
    \[
        \{\mathrm{mod} (l_0, \tau), \dots, \mathrm{mod}(l_{\tau-1}, \tau)\} = \{0, \dots, \tau-1\},
    \]
    the matrices $\{W^{(l_i)}\}_{i=0}^{\tau-1}$ satisfy the finite-time consensus property in \cref{def:fin-time-cons}.
\end{proposition}
\begin{proof}
    First, the symmetry of $W^{(l)}$ follows directly from its definition. To see the row stochastic property, we have
    \begin{align*}
        W^{(l)} \ones &= \left( W^{(l)}(p_{\tau-1}) \otimes W^{(l)}(p_{\tau-2}) \otimes \cdots \otimes W^{(l)}(p_0) \right) (\ones_{p_{\tau-1}} \otimes \cdots \otimes \ones_{p_0}) \\
        &= \left(W^{(l)}(p_{\tau-1}) \ones_{p_{\tau-1}} \right) \otimes \cdots \otimes \left( W^{(l)}(p_0) \ones_{p_0} \right) = \ones,
    \end{align*}
    where we use the mixed-product property of the Kronecker product. The column stochastic property follows from symmetry and row stochasticity. Hence, $W^{(l)}$ is doubly stochastic. Next, we show that the sequence $\{W^{(l)}\}_{l=0}^{\tau-1}$ has the finite-time consensus property. Utilizing the Kronecker product property, we establish that 
    \begin{align*}
        \prod_{l=0}^{\tau-1} W^{(l)} &= \prod_{l=0}^{\tau-1} \left( W^{(l)}(p_{\tau-1}) \otimes W^{(l)}(p_{\tau-2}) \otimes \cdots \otimes W^{(l)}(p_{0}) \right) \\
        &= \left(\prod_{l=0}^{\tau-1} W^{(l)}(p_{\tau-1}) \right) \otimes \left( \prod_{l=0}^{\tau-1} W^{(l)}(p_{\tau-2}) \right) \otimes \cdots \otimes \left( \prod_{l=0}^{\tau-1} W^{(l)}(p_0) \right) \\
        &= \left(\tfrac{1}{p_{\tau-1}} \ones_{p_{\tau-1}} \ones_{p_{\tau-1}}\tran \right) \otimes \cdots \otimes \left(\tfrac{1}{p_0} \ones_{p_0} \ones_{p_0}\tran \right) = \tfrac{1}{n} \ones \ones\tran,
    \end{align*}
    where we again use the mixed-product property. Combining this result with the fact that $p$-peer hyper-cuboids are periodic, and all instances of $p$-peer hyper-cuboids are commutative, we obtain that any permutation of any sequence of $W^{(l)}$ of length larger than $\tau$ has the finite-time consensus property as well.
\end{proof}

\begin{proposition}
    For any $n = 2^\tau$ with $\tau \in \natint_{\geq 1}$, the weight matrices of one-peer hyper-cubes defined in~\eqref{def:hypercube-mat} are symmetric, double-stochastic, and have the finite-time consensus property~\eqref{eq:fin-time-cons-cond}.
\end{proposition}
\begin{proof}
    When $n=2^\tau$, the one-peer hyper-cube is just a special case of $p$-peer hyper-cuboids. Decomposition of $n$ into its $(2, 2, \ldots, 2)$-based representation (or equivalently, binary representation) yields the desired properties.
\end{proof}

\paragraph{Connection with de Bruijn graphs}
The de Bruijn graph was first studied in \cite{Bruijn1946ACP}, and its finite-time consensus property has been established in \cite{delvenne2009optimal}. Here, we present a formulation of de Bruijn graphs closely related to $p$-peer hyper-cuboids~\eqref{def:hypercuboid-mat} and show that de Bruijn graphs are just permutations of $p$-peer hyper-cuboids.
\par
Let $n = p^\tau$ with $(p,\tau) \in \natint_{\geq 2} \times \natint_{\geq 1}$, and define the $p$-based representation of $i \in \{0,1,\ldots,n-1\}$ as $(i_{\tau-1} i_{\tau-2} \ldots i_0)_p$. Then, the $n \times n$ weight matrix of the de Bruijn graph is defined by 
\begin{equation} \label{def:de-bruijn-mat}
    w_{ij} = \begin{cases}
        \tfrac{1}{p} \quad &\text{if } (i_{\tau-2} i_{\tau-3} \ldots i_0)_p = (j_{\tau-1} j_{\tau-2} \ldots j_1)_p \\
        0 &\text{otherwise.}
    \end{cases}
\end{equation} 
The following proposition presents the connection between de Bruijn graphs and $p$-peer hyper-cuboids when the matrix size $n$ is a power of $p \in \natint_{\geq 2}$.
\begin{proposition} \label{prop:de-bruijn-connection}
    Given $n=p^\tau$ with $(p, \tau) \in \natint_{\geq 2} \times \natint_{\geq 1}$, let $W_\mathrm{db} \in \reals^{n \times n}$ be the weight matrix of the de Bruijn graph defined in~\eqref{def:de-bruijn-mat} and $\{W_\mathrm{hc}^{(l)}\}_{l \in \natint}$ be the weight matrices of the $p$-peer hyper-cuboids defined in~\eqref{def:hypercuboid-mat}. Then, for any $l \in \natint$, there exist $n \times n$ permutation matrices $P^{(l)}$ and $Q^{(l)}$ such that
    \begin{equation} \label{eq:de-bruijn-connection}
        W_\mathrm{hc}^{(l)} = P^{(l)} W_\mathrm{db} (Q^{(l)})\tran.
    \end{equation}
\end{proposition}
\begin{proof}
    See \cref{sec:app-de-bruijn}.
\end{proof}

\section{Algorithm description} \label{sec:alg-description}

This section presents Gradient Tracking with time-varying topologies (TV-GT) and our modified version, Gradient Tracking for Finite-Time Consensus Topologies (GT-FT). Henceforth, we will refer to each algorithm as TV-GT and GT-FT, respectively.

\subsection{Gradient tracking with time-varying topologies}

Gradient Tracking (GT) \cite{nedic2017achieving,di2016next} is a well-studied decentralized algorithm for solving the problem~\eqref{eq:prob}, and various formulations of TV-GT exist in the literature. The presented form follows from \cite{di2016next} (called Semi-ATC-TV-GT), and its implementation involves a sequence of graphs $\cG^{(k)} = (\cV, W^{(k)}, \cE^{(k)})$ which models the connections between the group of $n$ agents. Here, $\cV$ is the set of $n$ nodes and $\cE^{(k)} \subseteq \{ (i,j) \mid (i,j) \in \cV \times \cV\}$ describes the set of connections between agents. The set of agents $\cV$ remains static while the set of edges $\cE^{(k)}$ can be time-varying. The entry $w_{ij}^{(k)}$ in the matrix $W^{(k)}$ applies a weighting factor to the parameters sent from agent $i$ to agent $j$. If $w_{ij}^{(k)} = 0$, that means agent $i$ is not a neighbor of agent $j$ in $\cG^{(k)}$; \ie, $(i,j) \notin \cE^{(k)}$.
\par
Given an initial point $\{x_i^{(0)}\} \subset \reals^d$ and stepsize $\alpha \in \reals_{>0}$, set $g_i^{(0)} {=} \nabla F_i(x_i^{(0)}, \xi_i^{(0)})$ for $i=1,\ldots,n$. Then, TV-GT takes the following iterations for $k = 0,1,2, \ldots$
\begin{equation} \label{eq:algo-gt-recursion}
\begin{array}{ll}
    x_i^{(k+1)} &= \displaystyle \sum\limits_{{j \colon (j,i) \in \cE^{(k)}}} w_{ji}^{(k)} (x_j^{(k)} - \alpha g_j^{(k)}) \\ 
    g_i^{(k+1)} &= \displaystyle \sum\limits_{{j \colon (j,i) \in \cE^{(k)}}} w_{ji}^{(k)} g_j^{(k)} + \nabla F_i (x_i^{(k+1)}; \xi_i^{(k+1)}) - \nabla F_i (x_i^{(k)}; \xi_i^{(k)}),
\end{array}
\end{equation} 
where the iterates $g_i^{(k)}$ are often called the \textit{tracking variables}. The TV-GT iterations can be written in a more compact form, which relies on the augmented quantities
\begin{gather}
    \x^{(k)} \define \col\{x^{(k)}_1, \dots, x^{(k)}_n\} \in \reals^{dn}, \qquad
    \g^{(k)} \define \col\{g^{(k)}_1, \dots, g^{(k)}_n\} \in \reals^{dn}, \nonumber \\
    \xbar^{(k)} \define \frac{1}{n} \sum\limits_{i=1}^n x_i^{(k)}, \qquad
    \f(\x) \define \frac{1}{n} \sum\limits_{i=1}^n f_i(x_i), \qquad
    \overline{\nabla f} (\x^{(k)}) \define \frac{1}{n} \sum\limits_{i=1}^n \nabla f_i (x_i^{(k)}), \nonumber \\
    \bar\x^{(k)} \define \ones_n \otimes \xbar^{(k)}, \qquad
    \hat\x^{(k)} \define \x^{(k)} - \bar\x^{(k)}, \label{eq:xhat} \\
    \W^{(k)} \define W^{(k)} \otimes I_d \in \reals^{dn \times dn}, \qquad
    \nabla \f(\x) \define \col\{\nabla f_1(x_1), \dots, \nabla f_n(x_n) \} \in \reals^{dn}, \nonumber \\
    \nabla \F(\x; \bxi) \define \col\{\nabla F_1(x_1; \xi_1), \dots, \nabla F_n (x_n; \xi_n)\} \in \reals^{dn}. \nonumber
\end{gather}
With the augmented quantities, TV-GT~\eqref{eq:algo-gt-recursion} can be written compactly as
\begin{align*}
    \x^{(k+1)} &= \W^{(k)} ( \x^{(k)} - \alpha \g^{(k)} ) \\
    \g^{(k+1)} &= \W^{(k)} \g^{(k)} + \nabla \F (\x^{(k+1)}; \bxi^{(k+1)}) - \nabla \F (\x^{(k)}; \bxi^{(k)}).
\end{align*}
The choice of the matrix sequence $\{W^{(k)}\}$ is critical yet challenging. One focus of this work is to analyze TV-GT in which the network topologies are restricted to topology sequences that satisfy \cref{def:fin-time-cons}, and in particular, we focus on the nonconvex, stochastic setting. This scenario is not encapsulated by existing analysis of TV-GT; see \Cref{sec:related} for an in-depth discussion.

\subsection{Gradient tracking for finite-time consensus topologies}

We present the modified version of TV-GT in \cref{alg:gt}, where a graph sequence satisfying \cref{def:fin-time-cons} is used for communication. The update rules in GT-FT are known in GT literature. However, we are the first, to our knowledge, to propose and analyze this new scheme in which TV-GT is restricted to topology sequences that satisfy \cref{def:fin-time-cons}. TV-GT, as presented in other works (\eg,~\cite{nedic2017achieving, scutari2019distributed}), aims to be as general as possible when considering network topologies. In contrast, we aim to specialize the analysis of GT to leverage the largely unexploited finite-time consensus property. To differentiate between the two approaches, we call our scheme GT-FT.
\begin{algorithm}[!ht]\caption{Gradient Tracking for Finite-Time Consensus Topologies (GT-FT)} \label{alg:gt}
\begin{algorithmic}[1]
\State Agent $i$ input: $x_i^{(0)} \in \reals^d$ and stepsize $\alpha \in \reals_{>0}$.
\State Global input: The parameter $\tau \in \natint_{\geq 1}$ for finite-time consensus, and the sequence of matrices $\{W^{(l)}\}$ that satisfies Definition~\ref{def:fin-time-cons}.
\State Initialize $g_i^{(0)} = \nabla F_i(x_i^{(0)}, \xi_i^{(0)}) \in \reals^d$, $i=1,\ldots,n$.
\For{$k = 0, 1, 2, \ldots$}
    \For{$i = 1, \ldots, n$ (in parallel)}
        \begin{subequations}
        \begin{align}
            w_{ij}^{(k)} &=  W^{(\mathrm{mod} (k, \tau))}[i,j], \ \ \text{for all} \ j=1,\ldots,n \\
            x_i^{(k+1)} &= \sum\limits_{{j \colon (j,i) \in \cE^{(k)}}} {w_{ji}^{(k)}} (x_j^{(k)} - \alpha g_j^{(k)}) \label{eq:alg-x} \\
            g_i^{(k+1)} &= \sum\limits_{{j \colon (j,i) \in \cE^{(k)}}} {w_{ji}^{(k)}} g_j^{(k)} + \nabla F_i (x_i^{(k+1)}; \xi_i^{(k+1)}) - \nabla F_i (x_i^{(k)}; \xi_i^{(k)}). \label{eq:alg-g}
        \end{align}
        \end{subequations}
    \EndFor
\EndFor
\end{algorithmic}
\end{algorithm}
\par
Our analysis of GT-FT is presented in \Cref{sec:analysis}, and extensively uses the compact, networked form of \cref{alg:gt} (with the help of the augmented quantities):
\begin{subequations}
\begin{align}
    \W^{(k)} &= W^{(\mathrm{mod} (k, \tau))}  \otimes I_d \\ 
    \x^{(k+1)} &= \W^{(k)} (\x^{(k)} - \alpha \g^{(k)}) \label{alg:gt-x} \\
    \g^{(k+1)} &= \W^{(k)} \g^{(k)} + \nabla \F (\x^{(k+1)}; \bxi^{(k+1)}) - \nabla \F (\x^{(k)}; \bxi^{(k)}). \label{alg:gt-g}
\end{align}
\end{subequations}

\section{Algorithm analysis} \label{sec:analysis}

This section presents the convergence analysis of \cref{alg:gt}. The assumptions needed for the analysis are listed in \Cref{sec:analysis-ass}. In particular, we do not assume convexity, and we only have access to stochastic gradient estimates of each local function $f_i$. The convergence results are presented in \Cref{sec:analysis-result}, with detailed proofs postponed to \Cref{sec:app-analysis}.

\subsection{Assumptions} \label{sec:analysis-ass}

In this subsection, we list all the assumptions needed for analysis. Our analysis does not need convexity of the objective function and holds for general nonconvex problems in the form of~\eqref{eq:prob}. We start with the following assumption on the problem~\eqref{eq:prob}.
\begin{assumption} \label{ass:smooth}
Each function $f_i \colon \reals^d \to \reals$, $i=1,\ldots,n$ is continuously differentiable with an $L$-Lipschitz continuous gradient: there exists $L \in \reals_{>0}$ such that
\[
    \|\nabla f_i(x) - \nabla f_i(y)\| \leq L \|x-y\|, \ \ \text{for all} \ x,y \in \intr \dom f_i, \ \text{and for all} \ i = 1,\ldots,n.
\]
In addition, the objective function $f \colon \reals^d \to \reals$ is bounded below, and the optimal value of the problem~\eqref{eq:prob} is denoted by $f^\star \in \reals$.
\end{assumption}
\par
At each iteration of \cref{alg:gt}, a stochastic gradient estimator of each component function~$f_i$ is computed, based on the random variable $\xi_i^{(k)}$ in the probability space $(\Omega_i, \cF_i, \mathbb P_i)$. Given initial conditions, let~$\cF^{(0)}$ denote the $\sigma$-algebra corresponding to the initial conditions and, for all $k \in \natint_{\geq 1}$, let $\cF^{(k)}$ denote the $\sigma$-algebra defined by the initial conditions and the random variables $\{\x^{(1)}, \ldots, \x^{(k)}\}$, and denote the expectation conditioned on $\cF^{(k)}$ by $\Ex_k [\cdot] \define \Ex[\cdot \mid \cF^{(k)}]$.
\par
The following assumption is made on the stochastic gradient estimator.
\begin{assumption} \label{ass:noise}
For all $k \in \natint$ and for all $i=1,\ldots,n$, the random variables~$\xi_i^{(k)}$ are independent of each other. The stochastic gradient estimator satisfies
\[
    \Ex_k [\nabla F_i(x_i^{(k)}; \xi_i^{(k)})] = \nabla f_i(x_i^{(k)}), \quad \text{for all} \ i=1,\ldots,n.
\]
Moreover, there exists $\sigma \in \reals_{>0}$ such that for all $k \in \natint$ and $i = 1,\ldots,n$, it holds that
\[
    \Ex_k [\|\nabla F_i(x_i^{(k)}; \xi_i^{(k)}) - \nabla f_i (x_i^{(k)})\|^2] \leq \sigma^2.
\]
\end{assumption}

\subsection{Convergence analysis for \texorpdfstring{\cref{alg:gt}}{Algorithm~1}} \label{sec:analysis-result}

This section presents the convergence results of \cref{alg:gt}. To leverage the finite-time consensus property, we analyze the evolution of iterates over $\tau$ iterations and demonstrate that part of the consensus error vanishes after $\tau$ iterations. To see this, we rewrite the algorithm recursion \eqref{alg:gt-x}--\eqref{alg:gt-g} more compactly as
\[
    \begin{bmatrix} \x^{(k+1)} \\ \g^{(k+1)} \end{bmatrix} = 
    \begin{bmatrix} \W^{(k)} & - \alpha \W^{(k)} \\ 0 & \W^{(k)} \end{bmatrix}
    \begin{bmatrix} \x^{(k)} \\ \g^{(k)} \end{bmatrix} + 
    \begin{bmatrix} 0 \\ \nabla \F^{(k+1)} - \nabla \F^{(k)} \end{bmatrix},
\]
where we define $\nabla \F^{(k)} \define \nabla \F(\x^{(k)}, \xi^{(k)})$ for the ease of presentation. Then, we have 
\[ \resizebox{0.95\hsize}{!}{$\displaystyle
    \begin{bmatrix} \W^{(i)} & - \alpha \W^{(i)} \\ 0 & \W^{(i)} \end{bmatrix} \cdots
    \begin{bmatrix} \W^{(j)} & - \alpha \W^{(j)} \\ 0 & \W^{(j)} \end{bmatrix} =
    \begin{bmatrix} \coprod_{k=j}^i \W^{(k)} & - \alpha (i-j+1) \coprod_{k=j}^i \W^{(k)} \\ 0 & \coprod_{k=j}^i \W^{(k)} \end{bmatrix}.$}
\]
Consequently, the $\x$-update \eqref{alg:gt-x} can be written as
\[
    \x^{(k)} = \left(\coprod\limits_{i=k-1}^{0} \W^{(i)} \right) \x^{(0)} - \alpha \sum\limits_{j = 0}^{k-1} (k-j) \coprod\limits_{i=k-1}^{j} \W^{(i)} \big(\nabla \F^{(j)} - \nabla \F^{(j-1)} \big),
\]
where for convenience we set $\nabla \F^{(-1)} = 0$. Recall the finite-time consensus parameter $\tau \in \natint_{\geq 1}$ from \cref{def:fin-time-cons}, and assume $k - 1 \geq \tau$. Thus, at least one pass through the topology sequence satisfying \cref{def:fin-time-cons} has been performed. Then, we have
\begin{align*}
    \x^{(k)} &= \left(\coprod\limits_{i=k-1}^{0} \W^{(i)} \right) \hat{\x}^{(0)} - \alpha \sum\limits_{j = 0}^{k-1-\tau} (k-j) \left(\coprod\limits_{i=k-1}^{j} \W^{(i)} \right) \big(\nabla \F^{(j)} - \nabla \F^{(j-1)} \big) \nonumber \\
    &\phantom{=} \ - \alpha \sum\limits_{j = k - \tau}^{k-1} (k-j) \left(\coprod\limits_{i=k-1}^{j} \W^{(i)} \right) \big(\nabla \F^{(j)} - \nabla \F^{(j-1)} \big).
\end{align*}
Multiplying $\widehat{\I} \define I_{dn} - \tfrac{1}{n} \ones_n \ones_n \tran \otimes I_d$ on both sides yields 
\begin{align}
    \hat{\x}^{(k)} &\define \x^{(k)} - \bar\x^{(k)} = (I_{dn} - \tfrac{1}{n} \ones_n \ones_n \tran \otimes I_d) \x^{(k)} \nonumber \\
    &= \left(\coprod\limits_{i=k-1}^{0} \widehat{\W}^{(i)} \right) \hat{\x}^{(0)} - \alpha \sum\limits_{j = 0}^{k-1-\tau} (k-j) \left(\coprod\limits_{i=k-1}^{j} \widehat{\W}^{(i)} \right) \big(\nabla \F^{(j)} - \nabla \F^{(j-1)} \big) \nonumber \\
    &\phantom{=} \ - \alpha \sum\limits_{j = k - \tau}^{k-1} (k-j) \left(\coprod\limits_{i=k-1}^{j} \widehat{\W}^{(i)} \right) \big(\nabla \F^{(j)} - \nabla \F^{(j-1)} \big), \label{eq:lem-consensus-xhat-1-new}
\end{align}
where we use the definition of $\hat{\x}^{(k)}$ in \eqref{eq:xhat} and denote $\widehat{\W}^{(i)} \define \W^{(i)} - \tfrac{1}{n} \ones_n \ones_n \tran \otimes I_d$ for all $i \in \natint$. Recall that \Cref{alg:gt} cycles through a topology sequence satisfying \cref{def:fin-time-cons} repeatedly in the same order. So the term $\coprod_{i=k-1}^{j} \widehat{\W}^{(i)}$, for $j=0,1,\ldots,k-1-\tau$ include a products of $\widehat{\W}^{(i)}$ that consists of a complete cycle of a topology sequence satisfying~\cref{def:fin-time-cons}. Hence, by~\eqref{eq:fin-time-cons-residual}, the first two terms in~\eqref{eq:lem-consensus-xhat-1-new} are zero and $\hat\x^{(k)}$ in \eqref{eq:lem-consensus-xhat-1-new} reduces to
\begin{equation} \label{eq:lem-consensus-xhat-1}
    \hat{\x}^{(k)} = -\alpha \sum\limits_{j = k - \tau}^{k-1} (k-j) \coprod\limits_{i=k-1}^{j} \widehat{\W}^{(i)} \big(\nabla \F^{(j)} - \nabla \F^{(j-1)} \big),
\end{equation}
since by assumption $k - 1 \geq \tau$.

With the concise presentation of $\hat{\x}^{(k)}$, the rest of the analysis is typical in the literature and relies on two important inequalities. The \textit{descent inequality} helps establish the convergence property of the averaged iterates~$\bar x^{(k)}$. By exploiting the new presentation \eqref{eq:lem-consensus-xhat-1}, the \textit{consensus inequality} reveals the per-iteration behavior of the consensus error~$\hat \x^{(k)}$, and will be used to show that each agent's parameter $x_i^{(k)}$ converges to the average~$\bar x^{(k)}$.
\begin{lemma}[Descent inequality] \label{lem:descent}
Let \cref{ass:smooth,ass:noise} hold, let the mixing matrices $\{W^{(l)}\}_{l=0}^{\tau-1} \subset \reals^{n \times n}$ satisfy \cref{def:fin-time-cons}, and let the stepsize $\alpha$ satisfy $\alpha \in (0, \tfrac{1}{2L}]$. Then, the sequence $\{\x^{(k)}\}$ generated by \cref{alg:gt} satisfies
\[
    \resizebox{0.95\hsize}{!}{$\displaystyle \Ex \lVert \overline{\nabla f} (\x^{(k)}) \rVert^2 + \Ex \lVert \nabla f(\bar{x}^{(k)}) \rVert^2 \leq \frac{4}{\alpha} \left( \Ex \ftilde(\xbar^{(k)}) - \Ex \ftilde(\xbar^{(k+1)}) \right) + \frac{2 L^2}{n} \Ex \lVert \hat{\x}^{(k)} \rVert^2 + \frac{2\alpha L \sigma^2}{n},$}
\]
for all $k \in \natint$, where $\ftilde \define f - f^\star$.
\end{lemma}
The left-hand side of the inequality in \cref{lem:descent} is the (expected) gradient norm, which aligns with our main convergence result (see \cref{thm:GT-conv}). Such a convergence result is common in stochastic unconstrained optimization; see, \eg, \cite{bertsekas00gradient}, which analyzes (centralized) Stochastic Gradient methods (SGD). 
\par
The second lemma is on the consensus inequality and establishes that all agents' parameters converge to their average. 
\begin{lemma}[Consensus inequality] \label{lem:consensus}
Let \cref{ass:smooth,ass:noise} hold, let the mixing matrices $\{W^{(l)}\}_{l=0}^{\tau-1} \subset \reals^{n \times n}$ satisfy \cref{def:fin-time-cons}, and let the stepsizes satisfy $\alpha \in \big(0, \frac{1}{4 \sqrt{6} \tau^2 L} \big]$. Then, for all $T \in \natint_{\geq \tau}$, the sequence $\{\x^{(k)}\}$ generated by \cref{alg:gt} satisfies
\begin{align}
    \frac{1}{T+1} \sum_{k=0}^{T} \Ex \lVert \hat{\x}^{(k)} \rVert^2 &\leq \frac{2}{T+1} \sum_{k=0}^{\tau} \Ex \lVert \hat{\x}^{(k)} \rVert^2 + \frac{48 \alpha^4 \tau^4 n L^2}{T+1} \sum_{k=0}^T \Ex \lVert \nabla f(\bar{x}^{(k)}) \rVert^2 \nonumber \\
    &\phantom{\leq} \ + (12 \alpha^4 \tau^4 L^2 + 16 \alpha^2 \tau^3 n) \sigma^2. \label{eq:conv-consensus}
\end{align}
\end{lemma}
Note that the first summation term on the right-hand side of~\eqref{eq:conv-consensus} relies on the first $\tau+1$ iterates of~$\hat \x^{(k)}$, and recall that $\tau$ is a prescribed constant for \cref{alg:gt}. Hence, the term $\sum_{k=0}^{\tau} \Ex \lVert \hat{\x}^{(k)} \rVert^2$ can be viewed as a constant when we study the asymptotic behavior of the consensus error.
\par
The consensus inequality~\eqref{eq:conv-consensus} is used in tandem with the descent inequality to show that the consensus error in the result of \cref{lem:descent} vanishes asymptotically. Hence, not only does the averaged parameter asymptotically reach a stationary point of~\eqref{eq:prob} biased by stochastic noise, but all the agents' parameters also converge (because they reach a consensus). \Cref{thm:GT-conv} formally presents this result.
\begin{theorem} \label{thm:GT-conv}
    Let \cref{ass:smooth,ass:noise} hold, let the mixing matrices $\{W^{(l)}\}_{l=0}^{\tau-1} \subset \reals^{n \times n}$ satisfy \cref{def:fin-time-cons}, and let the stepsize satisfy $\alpha \in \big(0, \frac{1}{4 \sqrt{6} \tau^2 L} \big]$. Then, for all $T \in \natint_{\geq \tau}$, the sequence $\{\x^{(k)}\}$ generated by \cref{alg:gt} satisfies 
    \[
        \frac{1}{T+1}  \sum_{k=0}^{T} \left(\Ex \lVert \overline{\nabla f} (\x^{(k)}) \rVert^2+ \Ex \lVert \nabla f(\xbar^{(k)}) \rVert^2 \right)
        \leq \frac{\gamma_1}{\alpha T} + \frac{\gamma_2 L^2}{n T} + \frac{\gamma_3 \alpha^4 \tau^4 L^4 \sigma^2}{n} + \gamma_4 \alpha^2\tau^3 L^2\sigma^2 + \frac{\gamma_5 \alpha L \sigma^2}{n}
    \]
    with some constants $(\gamma_1, \gamma_2, \gamma_3, \gamma_4, \gamma_5) \in \reals_{>0}^5$.
\end{theorem}
The numeric constants $(\gamma_1, \gamma_2, \gamma_3, \gamma_4, \gamma_5)$, together with the proof of \cref{thm:GT-conv}, are provided in \Cref{sec:app-analysis}.
\par
Then, we tune the stepsize $\alpha$ using the technique described in \cite{koloskova2020unified, stich2019unified}.
Also, similarly to~\cite{NEURIPS2021_74e1ed8b}, we weaken the rate's dependence on $L$ by imposing a warm-up strategy (\eg, AllReduce \cite{assran2019stochastic}) to force all agents' parameters in the first $\tau+1$ iterations to be the same. This leads us to the following corollary.
\begin{corollary} \label{cor:warm-up}
    Let \cref{ass:smooth,ass:noise} hold, let $\{W^{(l)}\}_{l=0}^{\tau-1}\subset \reals^{n \times n}$ satisfy \cref{def:fin-time-cons}, and let
    \[
        \alpha = \min \big\{\big(\tfrac{1}{c_1 T}\big)^{\frac{1}{2}}, \big(\tfrac{1}{c_2 T}\big)^{\frac{1}{3}}, \tfrac{1}{4 \sqrt{6} \tau^2 L}\big\}, \quad \text{where} \ c_1 \define \tfrac{L \sigma^2}{n}, c_2 \define \tau^3 L^2 \sigma^2.
    \]
    Suppose a warm-up strategy (\eg, AllReduce) is applied so that all agents' parameters and tracking variables in the first period are the same: $\sum_{k=0}^{\tau} \Ex \lVert \hat\x^{(k)} \rVert^2 = 0$. Then, for all $T \in \natint_{\geq \tau}$, the sequence $\{\x^{(k)}\}$ generated by \cref{alg:gt} satisfies 
    \[
        \frac{1}{T+1} \sum_{k=0}^{T} \left( \Ex \lVert \overline{\nabla f} (\x^{(k)}) \rVert^2 + \Ex \lVert \nabla f(\bar{x}^{(k)}) \rVert^2\right) \leq \frac{\gamma_6 \tau^2 L}{T} + \gamma_7 \tau \left(\frac{L \sigma}{T}\right)^{\frac{2}{3}} + \gamma_8 \left(\frac{L \sigma^2}{n T}\right)^{\frac{1}{2}},
    \]
for some constants $(\gamma_6, \gamma_7, \gamma_8) \in \reals_{>0}^3$.
\end{corollary}

\subsection{Discussion and comparison with other analyses for TV-GT} \label{sec:analysis-discussion}

The derived convergence rate in \cref{cor:warm-up} depends on the finite-time consensus parameter $\tau$, and, remarkably, is independent of the connectivity of any of the individual elements of the graph sequence. This contrasts, for example, the rate derived in \cite{song2023communicationefficient}, which depends on the smallest connectivity (in expectation) of the set of all time-varying topologies used in the algorithm. Thus, their analysis cannot handle finite-time consensus graphs because some elements in such a deterministic topology sequence can be disconnected. The underlying reason for the incapability is that their analysis examines the ``worst'' topology in the sequence but fails to consider the joint effect of the entire topology sequence.
\par
Moreover, the superiority of GT-FT can be demonstrated via comparison with TV-GT restricted to the static counterparts of the finite-time consensus graph sequences. Given a sequence of graphs $\{\cG^{(l)} \equiv (\cV, W^{(l)}, \cE^{(l)})\}_{l=0}^{\tau-1}$, its static counterpart $\cG^{\mathrm{(static)}} = (\cV, W^{\mathrm{(static)}}, \cE^{\mathrm{(static)}})$ is defined by $\cE^\mathrm{(static)} = \cE^{(0)} \cup \cE^{(1)} \cup \cdots \cup \cE^{(\tau-1)}$, and the weight matrix $W^\mathrm{(static)}$ is normalized to be doubly stochastic. (More discussion on the static variants of a sequence of (dynamic) graphs can be found in \cite{Bruijn1946ACP, NEURIPS2021_74e1ed8b, harary1988hypercube, Shi_2016}.) Without loss of generality, we assume the number of agents is $n = 2^\tau$ for some $\tau \in \natint_{\geq 1}$, and a number of $\tau$ one-peer exponential graphs (satisfying \cref{def:fin-time-cons}) are used in GT-FT. In this case, the best existing rate \cite{song2023communicationefficient} (to our knowledge) of TV-GT depends on the spectral gap $1-\rho = 2/(1+\tau)$ of the static exponential graph \cite{NEURIPS2021_74e1ed8b}, and reads as
\[
    O \left(\frac{\sigma^2}{n\epsilon^2} \right) + O \left(\frac{(1 + \tau)^{\frac{3}{2}} \sigma}{\epsilon^{\frac{3}{2}}} \right) + O \left(\frac{(1 + \tau)^{2}}{\epsilon}\right),
\]
where the Lipschitz constant $L$ is omitted in \cite{song2023communicationefficient}. In comparison, it follows from \cref{cor:warm-up} that the iteration complexity of GT-FT using a sequence of $\tau$ one-peer exponential graphs is given by
\[
    O \left(\frac{L\sigma^2}{n\epsilon^2}\right) + O\left(\frac{\tau^{\frac{3}{2}}L\sigma}{\epsilon^{\frac{3}{2}}} \right) + O \left(\frac{\tau^2 L}{\epsilon} \right).
\]
Ignoring the Lipschitz constants as in \cite{song2023communicationefficient}, we find that this implementation of GT-FT has a similar iteration complexity to TV-GT using a static exponential graph. Remarkably, this slight improvement in convergence rate comes with a significant decrease in communication cost: The maximum degree of a static exponential graph is $\Theta (\log_2 n)$ while that of a single one-peer exponential graph is $\Theta (1)$. Similar comparisons can also be performed for $p$-peer hyper-cuboids and static hyper-cuboids.

\section{Numerical experiments} \label{sec:experiments}

In this section, we present numerical results to verify our theoretical findings. The purpose of the numerical experiments is two-fold. First, numerical evidence is provided to verify that the graph sequences studied in \Cref{sec:fin-time-cons} satisfy the finite-time consensus property. Moreover, we conduct numerical experiments that incorporate the studied graph sequences into decentralized optimization algorithms. The numerical results demonstrate that GT-FT using graph sequences with finite-time consensus property converges at the same rate as TV-GT using the static counterparts.

\subsection{Finite-time consensus property} \label{sec:exp-consensus}

We verify in numerical experiments that the presented graph sequences satisfy the finite-time consensus property. To do so, we simulate an average consensus problem. Each agent is initialized with a random vector~$x_i^{(0)} \sim \mathcal N(0, \Sigma)$ drawn from a Gaussian distribution (with $\Sigma \in \symm^d_{++}$). The iterates $x_i^{(k)}$ evolve according to the recursion $x_i^{(k+1)} = W^{(k)} x_i^{(k)}$ for $i=1,\ldots,n$, and the consensus error at each iteration is defined as
\[
    \Xi^{(k)} \define \frac{1}{n} \sum_{i=1}^n \|x_i^{(k)} - \xbar^{(0)}\|^2.
\]
In the simulation, we compare the presented time-varying graphs with their corresponding static variant, which is defined in \Cref{sec:analysis-discussion}.

The simulation results are presented in \Cref{fig:fin-time-consensus}, of which the setting and style closely follow \cite[Figure~4]{NEURIPS2021_74e1ed8b}. In \cref{fig:fin-time-consensus}, the graph sequences satisfying \cref{def:fin-time-cons} have a steep drop in the consensus error (see dashed lines), indicating the vanishing of the consensus error. In comparison, the consensus error for other graphs decreases asymptotically (at an exponential rate).
\begin{figure}[t]
    \centering
    \includegraphics[width=0.98\linewidth]{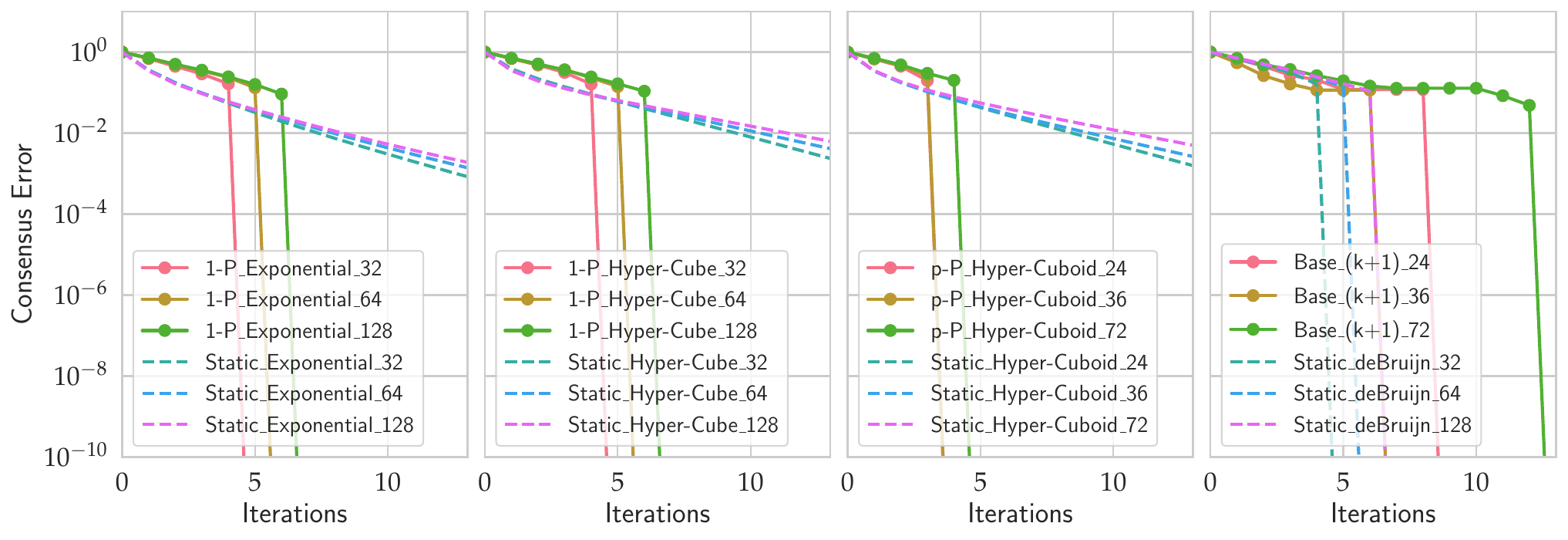}
    \caption{\small Consensus error versus the number of iterations. The legend is composed of three parts. The first part is either ``Static'', ``1-P'', or ``$p$-P'', standing for static graphs, one-peer time-varying graphs, and $p$-peer time-varying graphs, respectively. The second part of the legend describes the graph type: exponential, hyper-cube, de Bruijn, or hyper-cuboid. The third part is for the number of agents. All graphs satisfying Definition~\ref{def:fin-time-cons} are plotted with solid lines, while others are plotted with dashed lines. The static variants do not achieve finite-time consensus. The 1-P time-varying graphs only achieve finite-time consensus when the number of agents is a power of two. The $p$-P time-varying graphs always achieve finite-time consensus. The sizes of split network components used in Base-$(k+1)$ graphs \cite{takezawa2023exponential} are $(16, 8)$ for 24 nodes size, $(32, 4)$ for 36 nodes, and $(64, 8)$ for 72 nodes, respectively.}
    \label{fig:fin-time-consensus}
\end{figure}

\subsection{Gradient tracking with finite-time consensus topologies} \label{sec:exp-converge}

In this section, we provide numerical evidence to verify the theoretical guarantees established in \Cref{sec:analysis} and to demonstrate the potential benefits of the finite-time consensus property in decentralized optimization algorithms. More specifically, we apply GT-FT to solve the least squares problem with a nonconvex regularization term:
\[
    \mini \quad \frac{1}{n} \sum_{i=1}^n \lVert A_i x - b_i \rVert^2 + \mu \sum_{j=1}^{d} \frac{x[j]^2}{1+ x[j]^2},
\]
where the optimization variable is $x \in \reals^d$, $x[j]$ denotes the $j$th component of $x$, and the data $\{A_i, b_i\}$ is held exclusively by agent $i$. This problem instance has been used extensively in the literature (see, \eg, \cite{xin2021improved,alghunaim2021unified}), and we follow existing conventions to construct the problem data. The entries in each data matrix $A_i \in \reals^{m \times d}$ are drawn IID from the distribution $\cN(0,1)$, and so are the vectors $\{\tilde x_i\}_{i=1}^n \subset \reals^d$. The vector $b_i \in \reals^d$ is then computed by $b_i = A_i \tilde x_i + \delta z_i$, where $\delta \in \reals_{>0}$ is a prescribed constant and $z_i \in \reals^d$ is random noise with entries drawn IID from $\cN(0,1)$. In all the experiments, we set $m = 500$, $d=20$, and $\delta = 10$. The number of agents ($n$) might vary in the experiments and will be specified later.
\par
In addition, we consider the case where agents have access to the true gradients and the case where agents only have access to the stochastic gradient estimates. The stochastic gradient is formed by adding Gaussian noise to the true gradient, \ie, $\widehat{\nabla f}_i(x) = \nabla f_i(x) + s_i$ with $s_i \sim \cN (0, \sigma^2 I_d)$. The magnitude of the gradient noise can be controlled by the constant $\sigma^2$, and we set $\sigma^2 = 10^{-4}$. Regarding optimization algorithms, we consider GT-FT (\cref{alg:gt}) for graph sequences satisfying \cref{def:fin-time-cons}, TV-GT for the static counterparts, and DGD for both types of graphs. In all the experiments, the stepsize is set to $\alpha = 10^{-4}$.

\subsubsection{Numerical results}

\cref{fig:exp-exponential} presents the numerical results for one-peer exponential graphs and their static variant, and \cref{fig:exp-hc} presents those for $p$-peer hyper-cuboids and their static variant. Regardless of whether true gradients or their stochastic estimates are used, the convergence rate of decentralized algorithms using graph sequences with finite-time consensus is similar to that using their static counterparts. In view of this similar rate of convergence, using graph sequences satisfying \cref{def:fin-time-cons} would significantly reduce the communication costs compared with using the static counterpart. For example, the maximum degree of a static exponential graph is $\Theta (\log n)$ while that of a single one-peer exponential graph is $\Theta(1)$.

\begin{figure}
    \centering
    \begin{subfigure}[b]{0.45\textwidth}
    \centering
    \includegraphics[width=\textwidth]{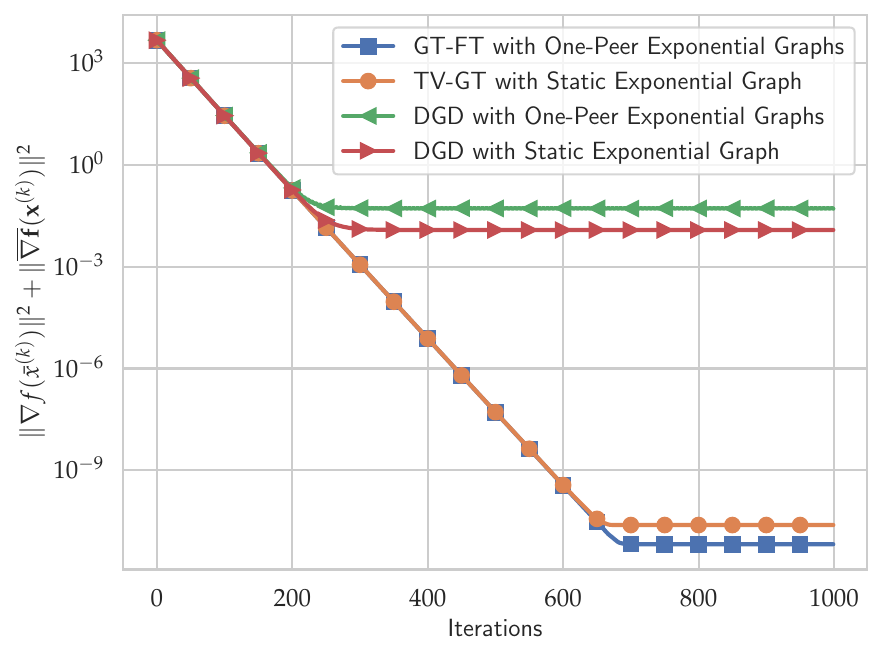}
    \caption{Agents use true gradients.}
    \label{fig:exp-exponential-determ}
    \end{subfigure}
    \begin{subfigure}[b]{0.45\textwidth}
    \centering
    \includegraphics[width=\textwidth]{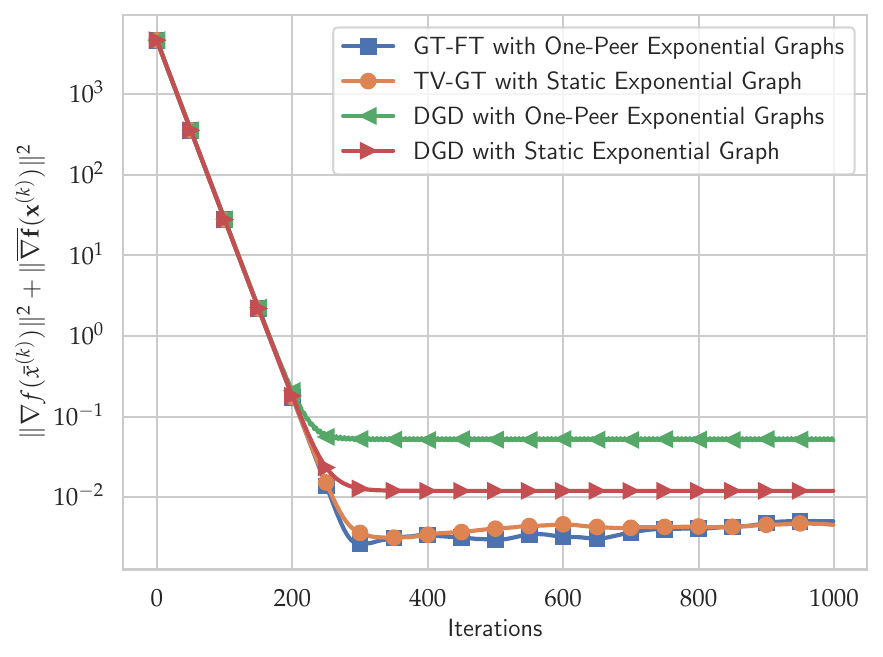}
    \caption{Agents use stochastic gradients.}
    \label{fig:exp-exponential-stoch}
    \end{subfigure}
    \caption{Comparison of the use of one-peer exponential graphs and static exponential graphs in decentralized algorithms. One-peer exponential graphs are used in GT-FT and DGD, and static exponential graphs are used in TV-GT and DGD.}
    \label{fig:exp-exponential}
\end{figure}

\begin{figure}
    \centering
    \begin{subfigure}[b]{0.45\textwidth}
    \centering
    \includegraphics[width=\textwidth]{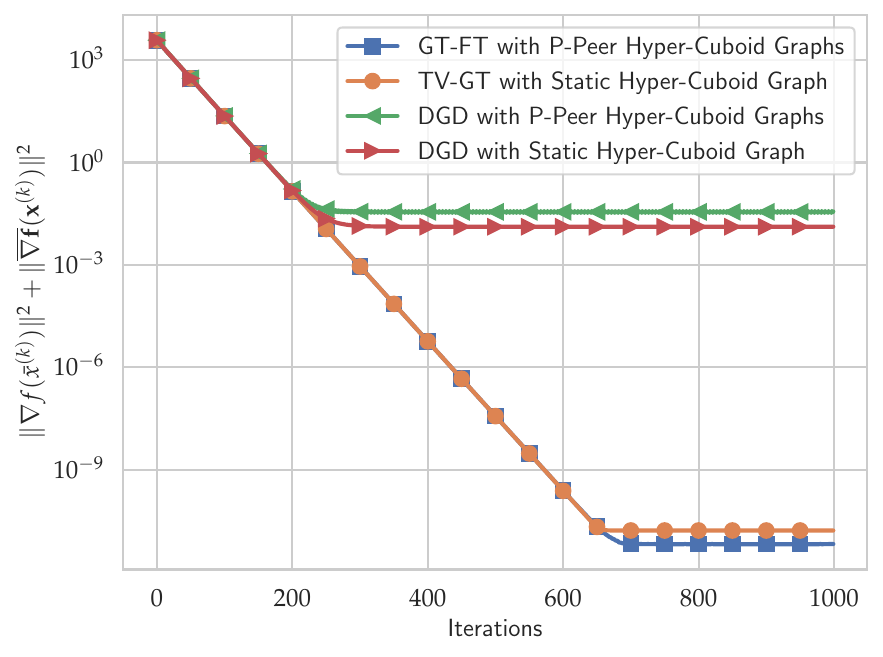}
    \caption{Agents use true gradients.}
    \end{subfigure}
    \begin{subfigure}[b]{0.45\textwidth}
    \centering
    \includegraphics[width=\textwidth]{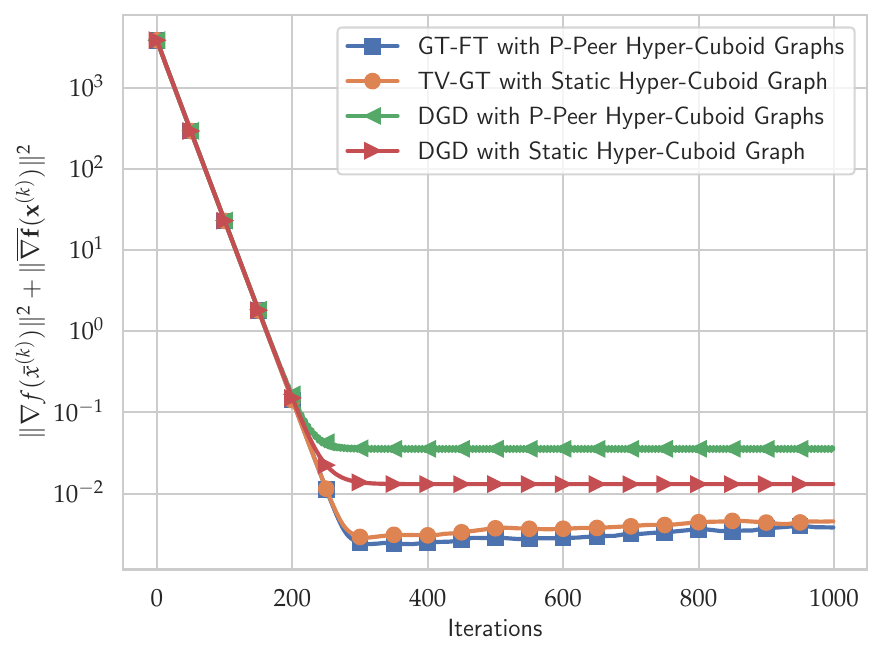}
    \caption{Agents use stochastic gradients.}
    \end{subfigure}
    \caption{Comparison of the use of $p$-peer hyper-cuboids and static hyper-cuboids in decentralized optimization algorithms. $p$-peer hyper-cuboids are used in GT-FT and DGD, and static hyper-cuboids are used in TV-GT and DGD.}
    \label{fig:exp-hc}
\end{figure}

\subsubsection{Discussion on GT-FT and DGD}

We conclude this section with a short discussion on the performance of GT-FT and DGD. When true gradients are used (and a constant stepsize is applied), GT-FT converges to a significantly better solution than DGD. The theoretical rationale for the better performance of GT methods has been thoroughly studied in the literature \cite{yuan2016convergence, alghunaim2019decentralized, koloskova2020unified, alghunaim2021unified}. With a constant stepsize, DGD converges to a sub-optimal solution biased proportionally to the magnitude of heterogeneity between agents~\cite{yuan2016convergence,koloskova2020unified}. In comparison, as a gradient tracking method, GT-FT is able to correct this bias caused by heterogeneity, and converges to a better solution \cite{alghunaim2019decentralized,alghunaim2021unified}. However, when agents can only use stochastic gradients, our experiments indicate that the difference in the quality of the solutions returned by GT-FT and DGD is much smaller. This is because both the gradient noise and the heterogeneity bias affect the computed solution returned by GT-FT and DGD. This phenomenon has already been observed for GT methods and DGD using a static topology~\cite{alghunaim2021unified}.

\section{Conclusions and future work}

We study several graph sequences satisfying the finite-time consensus property, including the one-peer exponential graphs, one-peer hyper-cubes, $p$-peer hyper-cuboids, and de Bruijn graphs. For each class of graphs, we present an explicit weight matrix representation and theoretically justify their finite-time consensus property. In particular, to the best of our knowledge, $p$-peer hyper-cuboids are one of the few classes of sparse graphs with arbitrary node sizes for which the finite-time consensus property is proven to hold. Moreover, we incorporate the studied topology sequences into the gradient tracking methods for decentralized optimization. Our analysis shows that the convergence rate of the proposed algorithmic scheme does not depend on the connectivity of any individual graph in the topology sequence, and the new scheme requires significantly lower communication costs compared with Gradient Tracking using the static counterpart of the topology sequence.
\par
Despite the success of incorporating graph sequences with finite-time consensus in decentralized optimization algorithms, several open questions remain. For example, the incorporation of the proposed graph sequences into other decentralized methods, such as EXTRA \cite{shi2015extra} and Exact Diffusion \cite{yuan2019exactdiffI}, is not straightforward. Furthermore, optimal communication costs and convergence rates for distributed algorithms with potentially disconnected time-varying graphs have yet to be established.


\appendix

\section{Supplementary materials for \texorpdfstring{\Cref{sec:fin-time-cons}}{Section~3}} \label{sec:app-de-bruijn}
We now present the proof of \cref{prop:de-bruijn-connection}, which establishes the connection between de Bruijn graphs and $p$-peer hyper-cuboids. The proof relies on the so-called \textit{perfect shuffle matrix} \cite[\S12.3]{horn13} and the Kronecker representation of de Bruijn graphs, which, to the best of our knowledge, has not been presented in the literature.
\begin{lemma}
    For $n = p^\tau$ with $(p, \tau) \in \natint_{\geq 2} \times \natint_{\geq 1}$ and for all vectors $\{a_l\}_{l=0}^{\tau-1}$, there exists a permutation matrix $P_{\mathrm s} \in \reals^{n \times n}$ such that
    \[
        P_{\mathrm s} (a_{\tau-1}\otimes a_{\tau-2} \otimes \cdots \otimes a_0) = a_{0}\otimes a_{\tau-1} \otimes \cdots \otimes a_{1}.
    \]
    Moreover, it holds that $P_{\mathrm s}^\tau = I_n$.
\end{lemma}
\begin{proof}
    The first result follows directly from \cite[\S12.3]{horn13}. The second result follows by repeatedly applying the first one:
    \[
        P_{\mathrm s}^\tau (a_{\tau-1} \otimes \cdots \otimes a_0) = P_{\mathrm s}^{\tau-1} (a_{0} \otimes a_{\tau-1} \otimes \cdots \otimes a_{1}) = \cdots = a_{\tau-1} \otimes a_{\tau-2} \otimes \cdots \otimes a_0. \qedhere
    \]
\end{proof}
This permutation matrix $P_{\mathrm s}$ is called the \textit{perfect shuffle matrix} \cite[\S12.3]{horn13}, hence the subscript ``s'' in $P_{\mathrm s}$. Following the same steps as in \eqref{eq:hypercuboid-prf-2}--\eqref{eq:hypercuboid-prf-4}, the matrix form of de Bruijn graphs can be written concisely as
\begin{equation} \label{def:de-bruijn-mat-kron}
    W_\mathrm{db} = \big(J_p \otimes I \otimes \cdots \otimes I \otimes I \big) P_{\mathrm s}\tran, \quad \text{where} \ J_p \define \tfrac{1}{p} \ones_p \ones_p\tran.
\end{equation}
\par
We now present the proof of \cref{prop:de-bruijn-connection}, which states the permutation equivalence between de Bruijn graphs and $p$-peer hyper-cuboids. More specifically, with the perfect shuffle matrix $P_{\mathrm s}$, the connection~\eqref{eq:de-bruijn-connection} stated in \cref{prop:de-bruijn-connection} can be written more explicitly as
\[
    W_\mathrm{hc}^{(\tau-l)} = (P_{\mathrm s}\tran)^{l+1} W_\mathrm{db} P_{\mathrm s}^l, \quad \text{for all} \ l=0,1,\ldots,\tau-1,
\]
where $W_\mathrm{hc}$ is the weight matrix of the $p$-peer hyper-cuboid~\eqref{def:hypercuboid-mat}.
\begin{proof}[Proof of \cref{prop:de-bruijn-connection}]
    It follows from \eqref{def:hypercuboid-mat-kron} that
    \begin{subequations}
        \begin{align}
            W_\mathrm{hc}^{(\tau-l)} &= \underbrace{I \otimes \cdots \otimes I}_{l-1 \text{ times}} \otimes J_p \otimes I \cdots \otimes I \label{eq:hc-deBruijn-prf-1} \\
            &= (P_{\mathrm s}\tran)^l (J_p \otimes I \otimes \cdots \otimes I \otimes I) P_{\mathrm s}^l \label{eq:hc-deBruijn-prf-2} \\
            &= (P_{\mathrm s}\tran)^l P_{\mathrm s}\tran P_{\mathrm s} (J_p \otimes I \otimes \cdots \otimes I \otimes I) P_{\mathrm s}^l \label{eq:hc-deBruijn-prf-3} \\
            &= P_{\mathrm s}\tran (P_{\mathrm s}\tran)^l W_\mathrm{db} P_{\mathrm s}^l. \label{eq:hc-deBruijn-prf-4} 
       \end{align}
    \end{subequations}
    Step~\eqref{eq:hc-deBruijn-prf-1} uses the Kronecker representation~\eqref{eq:hypercuboid-prf-4}. Step~\eqref{eq:hc-deBruijn-prf-2} applies the following property of the perfect shuffle matrix $l$ times to shift $J_p$ to the beginning: 
    \[
        P_{\mathrm s} (A_{\tau-1} \otimes A_{\tau-2} \cdots \otimes A_1 \otimes A_0) P_{\mathrm s}\tran = A_0 \otimes A_{\tau-1} \otimes \cdots \otimes A_2 \otimes A_1,
    \]
    where $\{A_i\}_{i=0}^{\tau-1}$ are any arbitrary $p \times p$ matrices. Step~\eqref{eq:hc-deBruijn-prf-3} uses the fact that $P_{\mathrm s}\tran P_{\mathrm s} = I$. Step~\eqref{eq:hc-deBruijn-prf-4} uses the definition of de Bruijn graphs in~\eqref{def:de-bruijn-mat-kron}.
\end{proof}

\section{Supplementary materials for \texorpdfstring{\Cref{sec:analysis}}{Section~5}} \label{sec:app-analysis}
This section includes the missing proofs from \Cref{sec:analysis}. First, we introduce the abbreviations 
\[
    \overline{\nabla f}^{(k)} \define \overline{\nabla f} (\x^{(k)}), \qquad
    \nabla \f^{(k)} \define \nabla \f(\x^{(k)}), \qquad
    \nabla \F^{(k)} \define \nabla \F (\x^{(k)}, \boldsymbol\xi^{(k)})
\]
for all $k \in \natint$, as well as a list of notations
\begin{gather}
    \bar{g}^{(k)} \define \frac{1}{n} \sum\limits_{i=1}^n g^{(k)}_i, \qquad \sbar^{(k)} \define \frac{1}{n} \sum\limits_{i=1}^n \big(\nabla F_i(\x^{(k)}_i; \xi^{(k)}_i) - \nabla f_i (\x^{(k)}_i) \big), \nonumber \\
    \bar{\g}^{(k)} \define \ones_n \otimes \bar{g}^{(k)}, \qquad
    \overline{\nabla \f}^{(k)} \define \overline{\nabla f}^{(k)} \otimes I_d, \qquad
    \s^{(k)} \define \nabla \F^{(k)} - \nabla \f^{(k)}. \label{eq:grad-noise}
\end{gather}
We also introduce the notation 
\[
    \coprod\limits_{k=i}^{j} \W^{(k)} \define \W^{(i)} \W^{(i-1)} \cdots \W^{(j)}, \ \ \text{for all} \ (i,j) \in \natint \times \natint \ \text{with} \ i \geq j.
\]

\subsection{Proof of \texorpdfstring{\cref{lem:descent}}{Lemma}}

The proof of the descent inequality is standard in the literature on GT methods. In particular, it follows from \cite[Eq.\ (73)]{alghunaim2021unified} that
\[
    \Ex f(\bar{x}^{(k+1)}) \leq f(\xbar^{(k)}) - \tfrac{\alpha (1 - \alpha L)}{2} \Ex \lVert \overline{\nabla f}^{(k)} \rVert^2 - \tfrac{\alpha}{2} \Ex \lVert \nabla f(\xbar^{(k)}) \rVert^2 \leq + \tfrac{\alpha}{2} \lVert \overline{\nabla f}^{(k)} - \nabla f(\xbar^{(k)}) \rVert^2 + \tfrac{\alpha^2 L \sigma^2}{2n},
\]
and from \cite[Eq.\ (66)]{alghunaim2021unified} that $\lVert \overline{\nabla f}^{(k)} - \nabla f(\bar{x}^{(k)}) \rVert^2 \leq \tfrac{L^2}{n} \lVert \x^{(k)} - \bar{\x}^{(k)} \rVert^2 = \tfrac{L^2}{n} \lVert \hat{\x}^{(k)} \rVert^2$.
Combining the above two inequalities yields the desired result.

\subsection{Proof of \texorpdfstring{\cref{lem:consensus}}{Lemma}} \label{sec:app-consensus}

The proof of the consensus inequality leverages the new presentation of $\hat{\x}^{(k)}$~\eqref{eq:lem-consensus-xhat-1}. More specifically, splitting~\eqref{eq:lem-consensus-xhat-1} via the definition~\eqref{eq:grad-noise} gives
\begin{align}
    \hat{\x}^{(k)} &= - \alpha \sum\limits_{j = k - \tau}^{k-1} (k-j) \coprod\limits_{i=k-1}^{j} \widehat{\W}^{(i)}  \big(\nabla \f^{(j)} - \nabla \f^{(j-1)} \big) - \alpha \widehat{\W}^{(k-1)} \s^{(k-1)} \nonumber \\
    &\phantom{=} \ - \alpha \sum\limits_{j=k-\tau}^{k-2} \Big( (k-j) \coprod\limits_{i=k-1}^{j} \widehat{\W}^{(i)} - (k-j-1)  \coprod\limits_{i=k-1}^{j+1} \widehat{\W}^{(i)} \Big) \s^{(j)}. \label{eq:lem-consensus-xhat-2}
\end{align}
To provide an upper bound on $\Ex_k [\|\hat\x^{(k)}\|^2]$, we bound each of the three terms on the right-hand side of~\eqref{eq:lem-consensus-xhat-2} one by one. First, it holds that
\begin{subequations}
\begin{align}
    \MoveEqLeft[0.2] \Ex_k \Big[\Big\lVert \sum\limits_{j=k-\tau}^{k-1} (k-j) \Big(\coprod\limits_{i=k-1}^j \widehat \W^{(i)} \Big) \big(\nabla \f^{(j)} - \nabla \f^{(j-1)} \big) \Big\rVert^2 \Big] \nonumber \\
    &\leq \tau^3 \sum\limits_{j=k-\tau}^{k-1} \Ex_k \Big[\Big\lVert \Big(\coprod\limits_{i=k-1}^j \widehat \W^{(i)} \Big) \big(\nabla \f^{(j)} - \nabla \f^{(j-1)} \big) \Big\rVert^2 \Big] \label{eq:lem-consensus-prf-1a} \\
    &\leq \tau^3 \sum\limits_{j=k-\tau}^{k-1} \Ex_k [\lVert \nabla \f^{(j)} - \nabla \f^{(j-1)} \rVert^2] \label{eq:lem-consensus-prf-1b} \\
    &\leq \tau^3 L^2 \sum\limits_{j=k-\tau}^{k-1} \Big(2\alpha^2 n \Ex_k [\lVert\nabla f(\xbar^{(j-1)})\rVert^2] + 9 \Ex_k [\lVert \hat\x^{(j-1)} \rVert^2] + 3 \Ex_k [\lVert \hat\x^{(j)} \rVert^2] + 3\alpha^2 \sigma^2 \Big). \label{eq:lem-consensus-prf-1c}
\end{align}
\end{subequations}
Note that \eqref{eq:lem-consensus-prf-1a} uses Jensen's inequality, and \eqref{eq:lem-consensus-prf-1b} uses the submultiplicative property of matrix norms and the fact that $\lVert \widehat \W^{(j)}\rVert_2 \leq 1$. The last step is standard in the analysis of GT methods (see, \eg, \cite[Lemma~11]{song2023communicationefficient}), and its derivation is omitted due to space limit.
\par
Second, it follows from the fact that $\lVert \widehat \W^{(k)}\rVert_2 \leq 1$, the definition of $\s^{(k)}$ \eqref{eq:grad-noise} and \cref{ass:noise} that
\begin{equation} \label{eq:lem-consensus-prf-2}
    \Ex_k [\|\widehat\W^{(k-1)} \s^{(k-1)}\|^2] \leq n \sigma^2, \ \ \text{for all} \ k \in \natint.
\end{equation}
\par
Third, it holds that
\begin{subequations}
\begin{align}
    \MoveEqLeft[0.2] \Ex_k \Big[\Big\lVert \sum\limits_{j=k-\tau}^{k-2} \Big( (k-j) \coprod\limits_{i=k-1}^{j} \widehat{\W}^{(i)} - (k-j-1)  \coprod\limits_{i=k-1}^{j+1} \widehat{\W}^{(i)} \Big) \s^{(j)} \Big\rVert^2 \Big] \nonumber \\
    &= \sum\limits_{j=k-\tau}^{k-2} \Ex_k \Big[\Big\lVert \Big( (k-j) \coprod\limits_{i=k-1}^{j} \widehat{\W}^{(i)} - (k-j-1)  \coprod\limits_{i=k-1}^{j+1} \widehat{\W}^{(i)} \Big) \s^{(j)} \Big\rVert^2 \Big] \label{eq:lem-consensus-prf-3a} \\
    &\leq 2 \sum\limits_{j=k-\tau}^{k-2} \Big( (k-j)^2 \Big\lVert \coprod\limits_{i=k-1}^j \widehat \W^{(i)} \Big\rVert_2^2 + (k-j-1)^2 \Big\lVert \coprod\limits_{i=k-1}^{j+1} \widehat \W^{(i)} \Big\rVert_2^2 \Big) \Ex_k [\lVert \s^{(j)} \rVert^2] \label{eq:lem-consensus-prf-3b} \\
    &\leq 4\tau^2 \sum\limits_{j=k-\tau}^{k-2} \Ex_k [\lVert \s^{(j)} \rVert^2] \label{eq:lem-consensus-prf-3c} \\
    &\leq 4 \tau^2 (\tau-1) n \sigma^2. \label{eq:lem-consensus-prf-3d}
\end{align}
\end{subequations}
In \eqref{eq:lem-consensus-prf-3a}, we use the independence of the gradient noise, and \eqref{eq:lem-consensus-prf-3b} uses the properties of matrix norms. Then, in \eqref{eq:lem-consensus-prf-3c} we use the fact $k-j-1 < k-j \leq \tau$ and that $\lVert {\widehat{\W}}^{(i)} \rVert^2 \leq 1$, and \eqref{eq:lem-consensus-prf-3d} applies \cref{ass:noise}.
\par
Combining \eqref{eq:lem-consensus-xhat-2}, \eqref{eq:lem-consensus-prf-1c}, \eqref{eq:lem-consensus-prf-2}, and \eqref{eq:lem-consensus-prf-3d} yields
\begin{align}
    \Ex_k [\lVert \hat{\x}^{(k)} \rVert^2] &\leq 12\alpha^4 \tau^3 n L^2 \sum\limits_{j=k-\tau}^{k-1} \Ex_k [\lVert \nabla f(\bar{x}^{(j-1)}) \rVert^2 ] + (6 \alpha^4 \tau^4 L^2 + 8 \alpha^2 \tau^3 n ) \sigma^2 \nonumber \\
    &\phantom{=} \ + 18\alpha^2 \tau^3 L^2 \sum\limits_{j=k-\tau}^{k-1} \Ex_k [\lVert \hat{\x}^{(j-1)} \rVert^2 ] + 6\alpha^2 \tau^3 L^2 \sum\limits_{j=k-\tau}^{k-1} \Ex_k [\lVert \hat{\x}^{(j)} \rVert^2]. \label{eq:lem-consensus-xhat-3}
\end{align}
where we also use Jensen's inequality and the fact that $\tau \geq 1$.  
\par
Recall our assumption at the beginning of \Cref{sec:app-consensus} that $k-1 \geq \tau$. So we sum up~\eqref{eq:lem-consensus-xhat-3} over iteration $k$ from $\tau + 1$ to $T$ ($T \geq \tau + 1$), add $\sum_{k=0}^{\tau} \Ex_k [\lVert \hat{\x}^{(k)} \rVert^2]$ and divide $(T+1)$ on both sides:
\begin{align}
    \MoveEqLeft[0.2] \frac{1}{T+1} \sum\limits_{k=0}^{T} \Ex_j [\lVert \hat{\x}^{(k)} \rVert^2] \nonumber \\ 
    &\leq \frac{12\alpha^4 \tau^3 n L^2}{T+1} \sum\limits_{k=\tau+1}^T \sum\limits_{j=k-\tau-1}^{k-2} \Ex_j [\lVert \nabla f(\bar{x}^{(j)}) \rVert^2 ] + \frac{1}{T+1} \sum\limits_{k=0}^{\tau} \Ex_k [\lVert \hat{\x}^{(k)} \rVert^2] \nonumber \\ 
    &\phantom{=} \ + \frac{18 \alpha^2 \tau^3 L^2}{T+1} \sum\limits_{k=\tau+1}^{T} \sum\limits_{j=k-\tau-1}^{k-2} \Ex_j [ \lVert \hat{\x}^{(j)} \rVert^2 ] + \frac{6\alpha^2 \tau^3 L^2}{T+1} \sum\limits_{k=\tau+1}^{T} \sum\limits_{j=k-\tau}^{k-1} \Ex_j [ \lVert \hat{\x}^{(j)} \rVert^2] \nonumber \\
    &\phantom{\leq} \ + \frac{(T - \tau + 1)}{T+1} (6 \alpha^4 \tau^4 L^2 + 8 \alpha^2 \tau^3 n)\sigma^2. \label{eq:lem-consensus-xhat-4}
\end{align}
Observe that for any constant $T \in \natint$ and for any sequence $\{\psi_j\} \subset \reals$, there exists a nonnegative sequence $\{\beta_j\} \subset \reals_{\geq 0}$ such that $\beta_j \leq 2\tau$ for $j=0,1,\ldots,T$ and
\begin{equation} \label{eq:lem-consensus-prf-aux} \textstyle
    \resizebox{0.86\hsize}{!}{$\displaystyle \sum\limits_{k=\tau+1}^{T} \sum\limits_{j = k - \tau -1}^{k-2} \psi_j = \sum\limits_{k=0}^{T-1} \beta_k \psi_k \leq 2\tau \sum\limits_{k=0}^{T} \psi_k, \;
    \sum\limits_{k=\tau+1}^{T} \sum\limits_{j = k - \tau}^{k-1} \psi_j = \sum\limits_{k=1}^{T} \beta_k \psi_k \leq 2\tau \sum\limits_{k=0}^T \psi_k.$}
\end{equation}
Thus, \eqref{eq:lem-consensus-xhat-4} can be further bounded as
\begin{align}
    \MoveEqLeft[0.2] \frac{1}{T+1} \sum\limits_{k=0}^{T} \Ex_k [\lVert \hat{\x}^{(k)} \rVert^2] \nonumber \\
    &\leq \frac{24\alpha^4 \tau^4 n L^2}{T+1} \sum\limits_{k=0}^T \Ex_k [\lVert \nabla f(\bar{x}^{(k)}) \rVert^2 ] + \frac{1}{T+1} \sum\limits_{k=0}^{\tau} \Ex_k [\lVert \hat{\x}^{(k)} \rVert^2] \nonumber \\ 
    &\phantom{=} \ \resizebox{0.74\hsize}{!}{$\displaystyle + \frac{48 \alpha^2 \tau^4 L^2}{T+1} \sum\limits_{k=0}^{T} \Ex_k [ \lVert \hat{\x}^{(k)} \rVert^2] + \frac{(T - \tau + 1)}{T+1} (6 \alpha^4 \tau^4 L^2 + 8 \alpha^2 \tau^3 n)\sigma^2$}, \label{eq:lem-consensus-xhat-5}
\end{align}
where we apply the first inequality in \eqref{eq:lem-consensus-prf-aux} twice (with $\psi_j \leftarrow \Ex_j [\lVert \nabla f(\bar{x}^{(j)}) \rVert^2]$ and with $\psi_j \leftarrow \Ex_j [\lVert \hat{x}^{(j)} \rVert^2]$), and the second inequality in \eqref{eq:lem-consensus-prf-aux} with $\psi_j \leftarrow \Ex_j [\lVert \hat{x}^{(j)} \rVert^2]$.  Then, the inequality \eqref{eq:lem-consensus-xhat-5} can be written equivalently as
\begin{align}
    \frac{1- 48 \alpha^2 \tau^4 L^2}{T+1} \sum\limits_{k=0}^{T} \Ex_k [\lVert \hat{\x}^{(k)} \rVert^2] &\leq \frac{24 \alpha^4 \tau^4 n L^2}{T+1} \sum\limits_{k=0}^T \Ex_k [\lVert \nabla f(\bar{x}^{(k)}) \rVert^2] + \frac{1}{T+1} \sum\limits_{k=0}^{\tau} \Ex_k [\lVert \hat{\x}^{(k)} \rVert^2] \nonumber \\
    &\phantom{\leq} \ + \frac{T - \tau + 1}{T+1} (6 \alpha^4 \tau^4 L^2 + 8 \alpha^2 \tau^3 n) \sigma^2. \label{eq:lem-consensus-xhat-6}
\end{align}
The stepsize condition $\alpha \leq \frac{1}{4 \sqrt{6} \tau^2 L}$ implies that $1 - 48 \alpha^2 \tau^4 L^2 \geq \tfrac{1}{2}$. Substituting this into~\eqref{eq:lem-consensus-xhat-6} gives
\begin{align*}
    \frac{1}{T+1} \sum\limits_{k=0}^{T} \Ex_k [\lVert \hat{\x}^{(k)} \rVert^2] &\leq \frac{48 \alpha^4 \tau^4 n L^2}{T+1} \sum\limits_{k=0}^T \Ex_k [\lVert \nabla f(\bar{x}^{(k)}) \rVert^2] + \frac{2}{T+1} \sum\limits_{k=0}^{\tau} \Ex_k [\lVert \hat{\x}^{(k)} \rVert^2] \nonumber \\
    &\textstyle \phantom{\leq} \ + \big(1 - \frac{\tau}{T+1} \big) (12 \alpha^4 \tau^4 L^2 + 16 \alpha^2 \tau^3 n) \sigma^2.
\end{align*}
Finally, the desired result \eqref{eq:conv-consensus} follows by taking the total expectation and relaxing $(1 - \tfrac{\tau}{T+1})$ to~$1$.

\subsection{Proof of \texorpdfstring{\cref{thm:GT-conv}}{Theorem}}

\cref{lem:descent} implies that for all $k \in \natint$,
\[ \textstyle
    \Ex \lVert \overline{\nabla f}^{(k)} \rVert^2 + \Ex \lVert \nabla f(\xbar^{(k)}) \rVert^2 \leq \frac{4}{\alpha} \big(\Ex \ftilde (\bar{x}^{(k)}) - \Ex \ftilde (\xbar^{(k+1)}) \big) + \frac{2 L^2}{n} \Ex \lVert \hat{\x}^{(k)} \rVert^2 + \frac{2\alpha L \sigma^2}{n},
\]
where $\ftilde (x) \define f - f^\star$. Taking the average over $k=0,\ldots,T$ gives
\[ 
    \frac{1}{T+1} \sum\limits_{k=0}^{T} \big(\Ex \lVert \overline{\nabla f}^{(k)} \rVert^2 {+} \Ex \lVert \nabla f(\xbar^{(k)}) \rVert^2 \big) \leq \frac{4\Ex \ftilde (\bar{x}^{(0)})}{\alpha (T+1)} + \frac{2 L^2}{n(T+1)} \sum\limits_{k=0}^T \Ex \lVert \hat{\x}^{(k)} \rVert^2 + \frac{2\alpha L \sigma^2}{n},
\]
where we use the fact $\tilde f(x) \geq 0$ for any $x \in \dom f$. Substituting~\eqref{eq:conv-consensus} yields
\begin{align*}
    \MoveEqLeft[0.2] \frac{1 - 96\alpha^4 \tau^4 L^4}{T+1} \sum\limits_{k=0}^{T} \big(\Ex \lVert \overline{\nabla f}^{(k)} \rVert^2+ \Ex \lVert \nabla f(\bar{x}^{(k)}) \rVert^2 \big) \\
    &\leq \frac{4}{\alpha (T+1)} \Ex \ftilde (\bar{x}^{(0)}) + \frac{4L^2}{n(T+1)} \sum\limits_{k=0}^\tau \Ex \lVert \hat{\x}^{(k)} \rVert^2 + \Big(\frac{24\alpha^4 \tau^4 L^4}{n} + 32 \alpha^2 \tau^3 L^2 + \frac{2\alpha L}{n} \Big) \sigma^2.
\end{align*}
The stepsize condition $\alpha \leq \frac{1}{4 \sqrt{6} \tau^2 L}$ implies $96\alpha^4 L^4 \tau^4 \leq \frac{1}{2}$, and thus we have
\begin{align}
    \MoveEqLeft[0.2] \frac{1}{T+1} \sum\limits_{k=0}^{T} \big(\Ex \lVert \overline{\nabla f}^{(k)} \rVert^2 + \Ex \lVert \nabla f(\bar{x}^{(k)}) \rVert^2 \big) \nonumber \\
    &\leq \resizebox{0.86\hsize}{!}{$\displaystyle \frac{8}{\alpha (T+1)} \Ex \ftilde (\bar{x}^{(0)}) + \frac{8L^2}{n(T+1)} \sum\limits_{k=0}^\tau \Ex \lVert \hat{\x}^{(k)} \rVert^2 + \Big(\frac{48\alpha^4 \tau^4 L^4}{n} + 64 \alpha^2 \tau^3 L^2 + \frac{4\alpha L}{n} \Big) \sigma^2.$} \label{eq:thm-prf} 
\end{align}
The desired result follows after further relaxing $\tfrac{1}{T+1}$ to $\tfrac{1}{T}$ and defining the constants $\gamma_1 = 8, \gamma_2 = 8 \sum_{k=0}^\tau \Ex \lVert \hat{\x}^{(k)} \rVert^2, \gamma_3 = 48, \gamma_4 = 64, \gamma_5 = 4$. (Note that $\gamma_2$ consists of only the first $\tau$ iterates, which can be safely viewed as a constant in an asymptotic analysis.)
    
\subsection{Proof of \texorpdfstring{\cref{cor:warm-up}}{Corollary}}

To simplify the notation, from now on, we use the notation $\lesssim$ to hide irrelevant constants. The notation $a \lesssim b$ means that there exists a positive constant $\gamma \in \reals_{>0}$ such that $a \leq \gamma b$. In our case, the important quantities that we keep are $\alpha$, $n$, $L$, and $\sigma$. For the statement of the corollary, we define the constants $\gamma_6=8, \gamma_7=64, \gamma_8 = 4$. Using this notation and the AllReduce warm-up strategy that guarantees $\sum_{k=0}^{\tau} ( \Ex \lVert \hat{\x}^{(k)} \rVert^2) = 0$, we obtain from~\eqref{eq:thm-prf} and the stepsize condition $\alpha \lesssim \tfrac{1}{\tau^2 L}$ that
\begin{align*}
    \MoveEqLeft[0.2] \frac{1}{T+1} \sum\limits_{k=0}^{T} \big(\Ex \lVert \overline{\nabla f}^{(k)} \rVert^2 + \Ex \lVert \nabla f(\bar{x}^{(k)}) \rVert^2 \big) \\
    &\textstyle \lesssim \frac{1}{\alpha T}  + \frac{\alpha^4 \tau^4 L^4 \sigma^2}{n} + \alpha^2 \tau^3 L^2 \sigma^2 + \frac{\alpha L \sigma^2}{n} \\
    &\textstyle \lesssim \frac{1}{\alpha T} + \frac{\alpha^2 L^2 \sigma^2}{n} + \alpha^2 \tau^3 L^2 \sigma^2 + \frac{\alpha L \sigma^2}{n} \\
    &\textstyle \lesssim \frac{1}{\alpha T} +  \alpha^2 \tau^3 L^2 \sigma^2 + \frac{\alpha L \sigma^2}{n} \\
    &\textstyle \lesssim \frac{1}{\alpha T} +  c_1 \alpha  + c_2 \alpha^2,
\end{align*}
where $c_1 \define \frac{L \sigma^2}{n}$ and $c_2 \define \tau^3 L^2 \sigma^2$. Now we set the stepsize $\alpha$ as
\[
    \alpha = \min \left\{ \left( \frac{c_0}{c_1 T} \right)^{\frac{1}{2}}, \, \left(\frac{c_0}{c_2 T} \right)^{\frac{1}{3}}, \, \frac{1}{4 \sqrt{6} \tau^2 L} \right\}.
\]
By definition, this choice of $\alpha$ satisfies the stepsize condition in \cref{thm:GT-conv}. We then discuss the following three cases.
\begin{enumerate}[label=(\alph*)]
    \item If $\alpha = \tfrac{1}{4\sqrt{6} \tau^2 L} \leq \min \{\big( \tfrac{1}{c_1 T} \big)^{\frac{1}{2}}, \big( \tfrac{1}{c_2 T} \big)^{\frac{1}{3}} \}$, then
    $\frac{1}{\alpha T}  + \alpha c_1  + c_2 \alpha^2  \lesssim \frac{1}{\alpha T} + \left(\frac{c_1}{T}\right)^{\frac{1}{2}} + \frac{c_2^{\frac{1}{3}}}{T^{\frac{2}{3}}}$.

    \item If $\alpha = \big( \tfrac{1}{c_1 T} \big)^{\frac{1}{2}} \leq \big(\tfrac{1}{c_2 T} \big)^{\frac{1}{3}}$, then
    $\frac{1}{\alpha T} + c_1 \alpha +c_2 \alpha^2 \lesssim \frac{ c_1^{\frac{1}{2}}}{T^\frac{1}{2}} + \frac{c_2^{\frac{1}{3}}}{T^{\frac{2}{3}}}$.

    \item If $\alpha = \big( \tfrac{1}{c_2 T} \big)^{\frac{1}{3}} \leq \big( \tfrac{1}{c_1 T} \big)^{\frac{1}{2}}$, then
    $\frac{1}{\alpha T} + c_1 \alpha + c_2 \alpha^2 \lesssim \frac{ c_2^{\frac{1}{3}}}{T^{\frac{2}{3}}} +  \left( \frac{c_1}{T} \right)^{\frac{1}{2}}$.
\end{enumerate}
Combining all three cases yields
\begin{align*}
    \textstyle \frac{1}{\alpha T} + c_1 \alpha + c_2 \alpha^2 &\textstyle \lesssim \frac{1}{\alpha T} + \left( \frac{c_2}{T^2} \right)^{\frac{1}{3}} + \left( \frac{c_1}{T} \right)^{\frac{1}{2}} \nonumber \\
    &\textstyle \lesssim \frac{\tau^2 L}{T} + \left( \frac{c_2}{T^2} \right)^{\frac{1}{3}} + \left( \frac{c_1}{T} \right)^{\frac{1}{2}} \\
    &\textstyle \lesssim \frac{\tau^2 L}{T} + \tau \left( \frac{L\sigma}{T}\right)^{\frac{2}{3}} + \left(\frac{L \sigma^2}{n T}\right)^{\frac{1}{2}}. 
\end{align*}

\end{document}